\documentclass[11pt]{amsart}
\usepackage{cite}
\usepackage{graphicx}
\usepackage{latexsym,amsmath,amsfonts,amscd, amsthm}
\usepackage{color}
\usepackage[colorlinks=true]{hyperref}
\hypersetup{urlcolor=blue, citecolor=red}
\usepackage{tikz}
\usetikzlibrary{arrows,backgrounds,snakes}
\usepackage{epsfig}
\usepackage{epstopdf}
\usepackage{dsfont}
\usepackage{amssymb}
\usepackage{xcolor}
\usepackage{comment}
\usepackage{subfigure}
\usepackage{multirow}
\usepackage{mathrsfs}
\usepackage{diagbox}
\usepackage{enumerate}

\newtheoremstyle{plainNoItalics}{}{}{\normalfont}{}{\bfseries}{.}{ }{}
\theoremstyle{plain}
\newtheorem{thm}{Theorem}[section]

\newtheorem{defn}[thm]{Definition}

\newtheorem{prop}[thm]{Proposition}
\newtheorem{exa}[thm]{Example}

\newcommand{\beq}{\begin{equation}}
\newcommand{\eeq}{\end{equation}}
\newcommand{\beqa}{\begin{eqnarray}}
\newcommand{\eeqa}{\end{eqnarray}}
\newcommand{\bit}{\begin{itemize}}
\newcommand{\eit}{\end{itemize}}
\newcommand{\bedef}{\begin{defn}}
\newcommand{\edefn}{\end{defn}}
\newcommand{\bpro}{\begin{prop}}
\newcommand{\epro}{\end{prop}}

\newcommand{\Dx}{\Delta x}
\newcommand{\Dy}{\Delta y}
\newcommand{\Dt}{\Delta t}

\newcommand{\eps}{\varepsilon}

\newcommand{\be}{{\bf e}}
\newcommand{\bx}{{\bf x}}

\newcommand{\bu}{{\bf u}}
\newcommand{\bn}{{\bf n}}
\newcommand{\bq}{{\bf q}}

\newcommand{\bU}{{\bf U}}

\newcommand\ds{ \displaystyle }

\title[SWEs with source terms]{High order asymptotic preserving and well-balanced schemes for the shallow water equations with source terms}

\keywords{Shallow water equations; Manning friction; asymptotic preserving; well-balanced; implicit-explicit Runge-Kutta; high order }

\begin{document}
	
\maketitle

\centerline{\scshape Guanlan Huang}
\medskip
{   \footnotesize
	\centerline{School of Mathematical Sciences, Xiamen University}
	\centerline{Xiamen, Fujian 361005, PR China}
	\centerline{glhuang@stu.xmu.edu.cn}
}

	\medskip

\centerline{\scshape Sebastiano Boscarino}
\medskip
{   \footnotesize
	\centerline{Department of Mathematics and Computer Science, University of Catania}
	\centerline{Catania 95125, Italy}
	\centerline{boscarino@dmi.unict.it}
}

	\medskip

\centerline{\scshape Tao Xiong } 
\medskip
{   \footnotesize
	\centerline{School of Mathematical Sciences, Fujian Provincial Key Laboratory of Mathematical Modeling} 
	\centerline{and High-Performance Scientific Computing, Xiamen University}
	\centerline{Xiamen, Fujian 361005, PR China}
	\centerline{txiong@xmu.edu.cn}
}

\bigskip

\begin{abstract}
	In this study, we investigate the Shallow Water Equations incorporating source terms accounting for Manning friction and a non-flat bottom topology.
	Our primary focus is on developing and validating numerical schemes that serve a dual purpose: firstly, preserving all steady states within the model, and secondly, maintaining the late-time asymptotic behavior of solutions, which is governed by a diffusion equation and coincides with a long time and stiff friction limit. 
 
	Our proposed approach draws inspiration from a penalization technique adopted in {\it{[Boscarino et. al, SIAM Journal on Scientific Computing, 2014]}}. By employing an additive implicit-explicit Runge-Kutta method, the scheme can ensure a correct asymptotic behavior for the limiting diffusion equation, without suffering from a parabolic-type time step restriction which often afflicts multiscale problems in the diffusive limit. Numerical experiments are performed to illustrate high order accuracy, asymptotic preserving, and asymptotically accurate properties of the designed schemes.
\end{abstract}

\vspace{0.1cm}

\section{Introduction}
\label{sec1}
\setcounter{equation}{0}
\setcounter{figure}{0}
\setcounter{table}{0}

The Shallow Water Equations (SWEs) are a set of equations derived from the more comprehensive Navier-Stokes equations, which govern the behavior of fluids. This hyperbolic system finds widespread application in modeling fluid dynamics in various natural environments, including channels, rivers, and oceans. 
The study of these equations carries significant practical implications, with numerous applications extending to crucial areas such as the prediction of storm surge levels, tides, and alterations in coastlines caused by events like hurricanes, among many others \cite{xing2005high,kurganov2002central,xing2006high,xing2006new,huang2022high}.

In recent years, the SWEs have become more and more important in atmospheric flows, especially with different space and time scales.
In this paper, we consider the SWEs with Manning friction and a non-flat bottom topology, which can be written as 
\begin{equation}
\left\{
 \begin{array}{ll}
  \partial_t h + \nabla \cdot \bq   = 0, \\ [3mm]
  \partial_t\bq + \nabla \cdot \left(\dfrac{\bq\otimes \bq}{h}\right) + \nabla\left(\dfrac{g}{2}h^2\right) = -gh\nabla b - \gamma \bq,
\end{array}\right.
\label{SWe_MB1}
\end{equation}
where $h$ is the depth of the water layer, $\bu=\bq/h$ is the flow velocity, which are defined on a time-space domain $(t,\bx)\in\mathbb{R}^+\times\Omega$; $g$ is the gravitational constant, $b(\bx)$ is the bottom topography which is independent of time; $\otimes$ denotes the Kronecker product, and $\gamma \bq$ models the bottom friction. 
Here $\gamma = \dfrac{gk^2|\bq|}{ h^{\eta}}$ with $|\bq|$ being the $L^2$ norm of $\bq$, and $k$ is the Manning coefficient to determine the intensity of the friction exerted by the bottom of the water.
The larger the $k$ is, the greater the friction exerted by the water. $\eta$ is a parameter which is usually taken to be $7/3$. 

To investigate the {\it late-time} behavior of solutions to the system (\ref{SWe_MB1}),  and following the lines of \cite{bulteau2020fully}, a behavior of $h$ and $\bq$ in a long time simulation and dominant friction is considered. Such an asymptotic behavior is governed by a diffusive regime. For the convenience of analysis, we introduce a small parameter $\eps$ in order to scale the time $t$ and the friction coefficient $k$, as follows: 
\begin{equation}
  \label{Dim_kt}
  t \leftarrow \frac{\hat{t}}{ \eps}\, ,\qquad {k\leftarrow\frac{\hat{k}}{\eps}}.\,
  \end{equation}
Then the system \eqref{SWe_MB1} can be rewritten as
\begin{equation}
\left\{
 \begin{array}{ll}
  \eps \partial_{\hat{t}}h + \nabla \cdot \bq   = 0, \\ [3mm]
\ds  \eps\partial_{\hat{t}}\bq + \nabla \cdot \left(\dfrac{\bq\otimes \bq}{h}\right) + \nabla\left(\dfrac{g}{2}h^2\right) = -gh\nabla b - \frac{1}{\eps^2}\widehat{\gamma} \bq,
\end{array}\right.
\label{SWe_MB2}
\end{equation}
with $\widehat{\gamma} = \dfrac{g\hat{k}^2|\bq|}{ h^{\eta}}$.
In \cite{BLeFT} (and the references therein), it was shown that, in the modeling of shallow water problems when the friction is particularly strong and very much dominates over the advection effects, the asymptotic regime, i.e. $\varepsilon \to 0$, associated with (\ref{SWe_MB2}) is governed by the nonlinear diffusion equation 
\begin{equation}\label{LIMIT}
h_t = \nabla\cdot \left( \frac{\sqrt{h}}{\kappa(h)} \frac{\nabla h}{\sqrt{|\nabla h|}}\right).
\end{equation}
In the next section, we show that by the Chapman-Enskog expansions of the variables $\bq$ and $h$ in term of the parameter $\varepsilon$, we get the limiting nonlinear diffusion equation (\ref{LIMIT}) when $\varepsilon \to 0$. 

Now, denoting 
\begin{equation}
\label{Dim_q}
\widehat{\bq} = \frac{\bq}{\eps},
\end{equation}
the system \eqref{SWe_MB2} becomes
\begin{equation}
\left\{
 \begin{array}{ll}
  \eps \partial_{\hat{t}}h + \eps\nabla \cdot \widehat{\bq}   = 0, \\ [3mm]
  \eps^2 \partial_{\hat{t}}\widehat{\bq} + \eps^2\nabla \cdot \left(\dfrac{\widehat{\bq}\otimes \widehat{\bq}}{h}\right) + \nabla\left(\dfrac{g}{2}h^2\right) = -gh\nabla b - \hat{\gamma} \widehat{\bq}.
\end{array}\right.
\label{SWe_MB3}
\end{equation}
Dividing the first sub-equation of system \eqref{SWe_MB3} by $\eps$, the second one by $\eps^2$,  and dropping the hats for clarity, then we can obtain the non-dimensional system
\begin{equation}
\left\{
 \begin{array}{ll}
  \partial_t h + \nabla \cdot \bq   = 0, \\ [3mm]
 \ds \partial_t\bq + \nabla \cdot \left(\dfrac{\bq\otimes \bq}{h}\right) + \frac{1}{\eps^2}\nabla\left(\dfrac{g}{2}h^2\right)
  = -\frac{1}{\eps^2}gh\nabla b - \frac{1}{\eps^2}\gamma \bq.
\end{array}\right.
\label{SWe_MB4}
\end{equation}
Notice that the nonlinear SWEs (\ref{SWe_MB4}) belong to the family of hyperbolic conservation laws with stiff source terms \cite{BLeFT}, i.e.,
\begin{equation}\label{ConsLaw}
U_t + \nabla\cdot F(U)  =	\frac{1}{\varepsilon^2}S(U),
\end{equation}
where $U=(h,\bq)^T$, $ \nabla\cdot F(U) = \left( \nabla \cdot \bq, \ds \nabla \cdot \left(\frac{\bq\otimes \bq}{h}\right)\right)^T$, and $S(U) = -\left(0,  \nabla\left(\dfrac{g}{2}h^2\right) +gh\nabla b + \gamma \bq\right)^T$. 

Solving (\ref{ConsLaw}) numerically is a challenging task for many reasons. Firstly, such system often admits non-trivial steady state solutions in which the source term balances the effect of the flux gradients. 
There are many interesting physical phenomena, which are small perturbations near these steady states. A standard numerical scheme cannot capture this equilibrium state or small perturbations near such equilibrium state, unless a much refined mesh is adopted \cite{xing2005high,wu2021uniformly,X2017}. 
A significant scheme for the SWEs with source terms is proposed by Bermudez and Vazquez \cite{bermudez1994upwind}, which balances the gradient of pressure by source terms at the discrete level, so that the equilibrium state can be maintained in coarse mesh subdivisions.
Then many schemes are designed to preserve the steady state \cite{leveque1998balancing,zhou2001surface,xing2005high,xing2006high,xing2006new,xing2006,CGP2006,ern2008well,noelle2006well,rhebergen2008discontinuous,Xing14,X2017,Kurganov18}, including finite difference, finite volume, and discontinuous Galerkin methods.
All the above works consider that the solutions are not affected by the friction, while there are many situations that the friction can not be ignored. 
The  additional friction term makes the well-balanced work more difficult, especially for situation that the friction is dominant. 
When the system \eqref{SWe_MB4} reaches steady state, it becomes 
\begin{equation}
  \left\{
   \begin{array}{ll}
    \nabla \cdot \bq   = 0, \\ [3mm]
 \ds   \nabla \cdot \left(\dfrac{\bq\otimes \bq}{h}\right) + \frac{1}{\eps^2}\nabla\left(\dfrac{g}{2}h^2\right)= -\frac{1}{\eps^2}gh\nabla b - \frac{1}{\eps^2}\gamma \bq.
  \end{array}\right.
  \label{SWe_MB6}
\end{equation}
In case of a still water $\bq=\mathbf{0}$, the system has a steady state named ``lake at rest'' \cite{yang2021high}, that is
\begin{equation}
	\label{Still-water}
	h(\bx,t) + b(\bx) = C, \qquad \bq(\bx,t) = \mathbf{0}.
\end{equation}
For a moving water steady state $\bq\ne \mathbf{0}$, it makes the design of a well-balanced scheme much more challenging. Very limited studies include traditional time discretizations \cite{michel2017well,gomez2021collocation}, or semi-implicit time discretizations \cite{song2011unstructured,xia2017efficient,xia2018new,chertock2015well}. 
However, in most of these works, the stiff friction is caused by the vanishing of water depth $h$.
Another stiff friction is caused by a large coefficient $k$ \cite{bulteau2020fully}, which is considered in this work. Besides, for simplicity, we only consider the still water steady state \eqref{Still-water}.

Secondly, a system as \eqref{ConsLaw} with a small parameter  $\varepsilon$ may relax to a reduced system as $\eps$ vanishes. Under suitable conditions \cite{CLL}, the solutions of the system (\ref{ConsLaw}) will converge to the solutions of the limiting system, see for example (\ref{LIMIT}). It is not at all obvious that a numerical scheme adopted for the solution of the system (\ref{ConsLaw}) will become (for given space and time discretization parameters)  a consistent discretization for the limiting equation, as the small parameter $\varepsilon$ vanishes. When this occurs, the scheme is said to be {\em asymptotic preserving} (AP). The AP scheme was proposed by S. Jin \cite{jin1999efficient} and we refer the readers to the review papers \cite{hu2017ap,jin2022asymptotic} for more work about AP schemes.
Recently, Bulteau et. al. \cite{bulteau2020fully} designed an AP scheme for system \eqref{SWe_MB4} by using an explicit time discretization, combined with finite volume and a well-balanced modification for space discretization. The scheme has been proven to be both AP and well-balanced. However, due to an explicit time discretization, the scheme might be subject to severe time step restrictions in order to maintain stability, as the eigenvalues of the system being $\lambda_1 = \bu\cdot{\bf{n}} + c/\eps$ and $\lambda_2 = \bu\cdot{\bf{n}} - c/\eps$ along the normal direction $\bn$, where $c=\sqrt{gh}$ representing the speed of sound. Such eigenvalues lead to a time step for an explicit scheme as
\begin{equation}
\Delta t = \text{CFL} \frac{\Delta x}{\max(|\bu|+c/\eps)}\sim \eps \Delta x,
\end{equation}
where $\Delta t$ is the time step size,  $\Delta x$ is the mesh size, and CFL is the time stability CFL number.

In the present paper, we will develop a numerical scheme to approximate the solution of \eqref{SWe_MB4}, which accurately captures the steady state solutions and  is able to restore the asymptotic regime, i.e., the scheme becomes a consistent discretization for the limiting equation (\ref{LIMIT}) as $\varepsilon \to 0$ (AP). Our main focus is to design a class of AP implicit-explicit (IMEX) high order finite difference weighted essentially non-oscillatory (WENO) schemes. However, the AP property only guarantees the consistency of the scheme, and in general does not guarantee the high order accuracy in the limit for $\varepsilon \ll 1$, i.e. the order of accuracy may degrade. We will show that under some assumptions, in the limit for $\varepsilon \to 0$, the scheme maintains its order of temporal accuracy, and we say that the scheme is {\em asymptotically accurate} (AA) for the limiting system when $\varepsilon \to 0$. To achieve high order temporal accuracy, IMEX Runge-Kutta (RK) schemes \cite{boscarino2014high,boscarino2016high} represent a powerful tool for the time discretization of stiff systems of the form (\ref{ConsLaw}). However, directly applying IMEX RK schemes to the SWEs (\ref{SWe_MB4}) are not very efficient. In the diffusive regime, i.e. $\varepsilon \to 0$, such schemes become consistent but explicit schemes for the diffusive limiting equation (\ref{LIMIT}). Therefore, these schemes suffer from a usual parabolic time step restriction $\Delta t = O(\Delta x^2)$, for given space and time discretization parameters $\Delta x$ and $\Delta t$. 
In order to avoid this parabolic time-step restriction, we introduce a new reformulation of the problem (\ref{SWe_MB4}) based on the addition of two opposite diffusive terms in the first equation in (\ref{SWe_MB4}), which has been introduced in \cite{boscarino2013flux, boscarino2013implicit, boscarino2014high,xiong2022high}, and the resulting scheme is consistent and relaxed to an implicit scheme in the diffusive limit, therefore avoiding the typical parabolic restriction. This strategy is very effective. Furthermore, we will show that, under this reformulation of the system (\ref{SWe_MB4}), the scheme converges to a high order implicit scheme in the diffusive limit (\ref{LIMIT}), namely the scheme is AA.
 
The rest of this paper is organized as follows.
In Section \ref{sec2}, we will formally derive the limiting equation for the SWEs with a non-flat topography and Manning friction. An IMEX time discretization and finite difference WENO reconstruction with well-balanced property will also be introduced. Formal analyses of AP and AA properties of the scheme are provided in Section \ref{sec4}. Numerical experiments are performed in Section \ref{sec5}, followed by concluding remarks in Section \ref{sec6}. 

\section{Diffusive limit and numerical scheme }
\label{sec2}
\setcounter{equation}{0}
\setcounter{figure}{0}
\setcounter{table}{0}

In this section, we will first briefly review the diffusive limit for the SWEs with strong friction and a bottom topography, which leads to a nonlinear degenerate $p$-Laplacian equation. Then we will present an AP scheme, which can well capture this nonlinear diffusive limit. This approach consists of adding and subtracting the limiting nonlinear diffusive term in the water depth equation, with different time discretizations and we will show that in the diffusive limit the numerical scheme becomes an implicit time discretization for the limiting system, which can avoid a parabolic time step restriction. This technique has been developed in \cite{boscarino2013flux, boscarino2013implicit, boscarino2014high} and used in \cite{peng2020stability,xiong2022high}. 

For this system, since all terms with explicit and implicit time discretizations are seperated, we employ an additive implicit-explicit (IMEX) scheme, with a multi-stage RK method to achieve high order accuracy. In space, we adopt a high order finite difference WENO reconstruction for convective terms \cite{shu1998essentially,jiang1996efficient}, combined with central difference discretizations for second order derivative terms. With a careful treatment in space, the well-balanced property for a hydrostatic equilibrium state is also achieved. For our proposed scheme, we can formally prove it is AP and AA.

\subsection{Diffusive limit}

Firstly, let us briefly review the limiting system for the SWEs
with friction and a non-flat bottom \eqref{SWe_MB4}, under the diffusive scaling with \eqref{Dim_kt} and \eqref{Dim_q}. We denote $H=h+b$ as the surface level of water, and assume the following Chapman-Enskog expansions  \cite{bulteau2020fully}
\begin{equation}
	\label{Exp1}
	\left\{
	\begin{array}{ll}
		h(\bx,t) = h_0(\bx,t) +\eps h_1(\bx,t) +\eps^2 h_2(\bx,t) +  \cdots,\\ [3mm]
		\bq(\bx,t) = \bq_0(\bx,t)+ \eps \bq_1(\bx,t) + \eps^2 \bq_2(\bx,t) + \cdots,
	\end{array}\right.
\end{equation}
with
\begin{equation}
\label{Exp2}
H(\bx,t) = h_0(\bx,t) + b(\bx) +\eps h_1(\bx,t) +\eps^2 h_2(\bx,t) +  \cdots.
\end{equation}
Here the bottom topography $b(\bx)$ is independent of time. Substituting \eqref{Exp1} into \eqref{SWe_MB4}, for the leading order terms, we have
\begin{subequations}
    \begin{align}
        &\partial_t h_0 + \nabla \cdot \bq_0   = 0,
        \label{lim_E1}\\[3mm]
        &\nabla\left(\dfrac{g}{2}h_0^2\right) +  gh_0\nabla b =  - \gamma_0 \bq_0,
        \label{lim_E2}
    \end{align}
\end{subequations}
with the nonlinear friction coefficient $\gamma_0 = gk^2|\bq_0|/h_0^{\eta}$. Denoting $H_0(\bx,t) = h_0(\bx,t) + b(\bx)$, since 
\begin{equation}
\nabla\left(\dfrac{g}{2}h_0^2\right)+gh_0\nabla b=gh_0\nabla h_0+gh_0\nabla b=gh_0\nabla H_0,
\label{eq2.4}
\end{equation} 
with $\gamma_0 = gk^2|\bq_0|/h_0^{\eta}$, substituting \eqref{eq2.4} into \eqref{lim_E2} we obtain
\begin{equation}
gk^2|\bq_0|^2/h_0^\eta=|\gamma_0\bq_0|=gh_0|\nabla H_0|, 
\end{equation}
which yields $|\bq_0|=\sqrt{h_0^{\eta+1}|\nabla H_0|/k^2 }$, and 
\eqref{lim_E2} becomes
\begin{equation}
\bq_0 = -\sqrt{\frac{h_0^{\eta+1}}{k^2}}\frac{\nabla H_0}{\sqrt{|\nabla H_0|}}.
\label{lim_E3}
\end{equation}
Substituting \eqref{lim_E3} into \eqref{lim_E1}, it leads to the following limiting equation
\begin{equation}
	\partial_t h_0 = \nabla\cdot\left( \sqrt{\frac{h_0^{\eta+1}}{k^2}}\frac{\nabla H_0}{\sqrt{|\nabla H_0|}}\right). \label{lim_E4-1}
\end{equation}
This limiting equation \eqref{lim_E4-1} is a non-stationary $p$-Laplacian equation with $p=3/2$. $p$-Laplacian equations are degenerate parabolic equations, which arise in many physical problems \cite{diaz1981estimates}, and they have been studied both theoretically and numerically in several works, such as \cite{andreu1999existence,herrero1981asymptotic,kamin1988fundamental,de1999regularity,eom2020large,Crispo2019global,garcia1987existence,barrett1993finite,barrett1994finite}, and references therein. Due to the limiting equation itself is already nonlinear and degenerate, it makes the designing of an efficient numerical scheme for \eqref{SWe_MB4} which can capture the diffusive limit \eqref{lim_E4-1} very challenging. In the following, we will propose an efficient AP scheme, which is high order both in space and in time and can well capture the corresponding limit.

\subsection{IMEX time discretization}
We consider an IMEX time discretization for \eqref{SWe_MB4}, which only uses implicit time discretizations for stiff and diffusion terms, while convective terms are discretized explicitly, so that we can achieve a large time step condition, while avoiding to solve a very complicate nonlinear system. We start with writing \eqref{SWe_MB4} in the following form of ordinary differential equations 
\begin{equation}
U_t =  F(U) + G(U),
\label{Add_IMEX_PDE}
\end{equation}
with $U=(h,\bq)^T$ and
\begin{equation}\label{split-AD}
F(U) =
\begin{pmatrix}
	-\nabla \cdot \bq - \mu(\eps)\nabla\cdot\left( \sqrt{\frac{h^{\eta+1}}{k^2}}\frac{\nabla H}{\sqrt{|\nabla H|}}\right)  \\ \, \\
	- \nabla \cdot \left(\dfrac{\bq\otimes \bq}{h}\right)
	\end{pmatrix},\,
G(U) =
\begin{pmatrix}
	 \mu(\eps) \nabla\cdot\left( \sqrt{\frac{h^{\eta+1}}{k^2}}\frac{\nabla H}{\sqrt{|\nabla H|}}\right)  \\ \, \\
	 - \dfrac{1}{\eps^2}\nabla\left(\dfrac{g}{2}h^2\right)  -\dfrac{1}{\eps^2}gh\nabla b - \dfrac{1}{\eps^2}\gamma \bq
	\end{pmatrix},
\end{equation}
where we have added and substracted a term for the equation of water depth $h$, based on the limiting equation (\ref{lim_E4-1}). Here the weighted function $\mu(\varepsilon)$ is such that $\mu:\mathbb{R}^+ \to [0,1]$ and $\mu(0)=1$. When $\varepsilon$ is not small, there is no need to add and subtract the limiting term, therefore $\mu(\varepsilon)$ will be small in such a hyperbolic regime, i.e., $\mu(\varepsilon) \approx 0$. The precise choice of $\mu(\varepsilon)$ depends not only on $\eps$ but also the mesh size or the time step as discussed in \cite{boscarino2014high}, which will be specified later.
Now let us focus on the time discretization while keeping the space continuous. The spatial discretizations will be discussed in the next section.
\subsubsection{\bf{First order IMEX scheme}} 
We begin with a first order IMEX scheme for (\ref{Add_IMEX_PDE}) with (\ref{split-AD}). The scheme reads
\begin{subequations}
	\begin{align}
		&\dfrac{h^{n+1} - h^n}{\Delta t}+ \nabla \cdot \bq^n +\mu(\eps) \nabla\cdot\left( \sqrt{\dfrac{h^{\eta+1}}{k^2}}\dfrac{\nabla H}{\sqrt{|\nabla H|}}\right)^n =  \mu(\eps)\nabla\cdot\left( \sqrt{\dfrac{h^{\eta+1}}{k^2}}\frac{\nabla H}{\sqrt{|\nabla H|}}\right)^{n+1}, \label{SWE_s3-1}\\[3mm]
		&\dfrac{\bq^{n+1} -\bq^n}{\Delta t} + \nabla \cdot \left(\dfrac{\bq\otimes \bq}{h}\right)^n + \dfrac{1}{\eps^2}\nabla\left(\dfrac{g}{2}h^2\right)^{n+1} = -\dfrac{1}{\eps^2}gh^{n+1}\nabla b - \dfrac{1}{\eps^2}\gamma^{n+1} \bq^{n+1}. \label{SWE_s3-2}
	\end{align}
	\label{SWE_s3}
\end{subequations}
Here we take the weighted function to be $\mu(\eps)=e^{-\eps^2/\Delta x}$ with $\Dx$ being the mesh size \cite{boscarino2013flux, boscarino2013implicit}.
This weighted function depends on $\eps$ and $\Dx$, which ensures that as $\eps$ goes to zero, $\mu(\eps)$ will approach $1$, while it approaches $0$ when $\eps$ is of $\mathcal{O}(1)$ and $\Dx$ is small enough. In this case, as $\eps\rightarrow 0$, the second and third terms on the left hand side cancel with each other, and the scheme \eqref{SWE_s3-1} formally converges to 
\begin{equation}
\frac{h^{n+1} - h^n}{\Delta t} - \nabla\cdot\left\{ \sqrt{\frac{h^{\eta+1}}{k^2}}\frac{\nabla H}{\sqrt{|\nabla H|}}\right\}^{n+1}=0,
\end{equation}
which is an implicit time discretization for the limting equation \eqref{lim_E4-1}, so that a parabolic type time discretization namely $\Dt = \mathcal{O}(\Dx^2)$ can be avoided. 

However, solving \eqref{SWE_s3} we face two important issues. Let us 
rewrite \eqref{SWE_s3-1} as
\begin{equation}
\label{IMEX1_ellip1}
h^{n+1} = h_{\star} + \Dt\mu(\eps)\nabla\cdot\left\{ \sqrt{\frac{h^{\eta+1}}{k^2}}\frac{\nabla H}{\sqrt{|\nabla H|}}\right\}^{n+1},
\end{equation}
with
\begin{equation*}
h_{\star} = h^n - \Dt\left\{\nabla \cdot \bq^n +\mu(\eps) \nabla\cdot\left( \sqrt{\frac{h^{\eta+1}}{k^2}}\frac{\nabla H}{\sqrt{|\nabla H|}}\right)^n\right\}.
\end{equation*}
Since $H^{n+1} = h^{n+1} + b$, \eqref{IMEX1_ellip1} is a fully nonlinear system of $h^{n+1}$, especially the diffusion coefficient depends on $\sqrt{|\nabla H|}$ in the denominator. The first issue is that when the water is at or close to rest, we have $\nabla H = 0$ or $\nabla H$ close to $0$. In this case, this nonlinear diffusion term cannot be linearized by using the approach in \cite{xiong2022high}. Instead we have to take $\nabla H$ in the numerator and denominator of \eqref{IMEX1_ellip1} to be at the same time level to solve a nonlinear system, in order to avoid such a singularity in the denominator. We adopt a simple Picard iteration to solve this nonlinear system as follows:
\begin{itemize}
\item  Starting from the initial guess $h^{n+1,0}=h^n$, at each iteration point, we first compute $|\nabla H|$ (based on a numerical approximation), if $|\nabla H|<\delta$ where $\delta$ is a small value (e.g., we take $\delta = 10^{-9}$ in our numerical tests, unless otherwise specified), we stop the iteration and take $h^{n+1}=h^{n+1,0}=h^n$; 

\item Otherwise, we continue the iteration by taking 
\begin{equation}
h^{n+1,\ell+1} = h_{\star} +  \Dt\mu(\eps)\nabla\cdot\left\{\frac{ \left(\sqrt{\frac{h^{\eta+1}}{k^2}}\right)^{n+1,\ell}}{\sqrt{|\nabla H|^{n+1,\ell}}} \nabla H^{n+1,\ell+1}\right\}, \quad \ell \ge 0.
\end{equation}
The iteration is stopped when $\| h^{n+1,\ell+1} - h^{n+1,\ell}\| \le \delta $. Here $\delta$ is taken to be the same as above, and an $L^1$ norm is taken in practice.  
\end{itemize} 
By taking this simple Picard iteration, as we can see, when initially the solution is at a hydrodynamic equilibrium, that is $H=\text{Const.}$, due to the iteration is stopped at the first step, such an equilibrium can be preserved.

After obtaining $h^{n+1}$, the second issue is how to update $\bq^{n+1}$ by \eqref{SWE_s3-2}, since the friction coefficient $\gamma^{n+1} = \dfrac{gk^2|\bq^{n+1}|}{(h^{n+1})^{\eta}}$ also depends on $\bq^{n+1}$. We first rewrite \eqref{SWE_s3-2} as
\begin{equation}
\label{SWE_s4}
\left(1 + \frac{\Dt }{\eps^2}\,\frac{gk^2|\bq^{n+1}|}{(h^{n+1})^{\eta}}\right)\bq^{n+1}
= \bq_{\star},
\end{equation}
here
\begin{equation}
	\label{qstar}
  \bq_{\star} =  \bq^n - \Dt\nabla \cdot \left(\frac{\bq\otimes \bq}{h}\right)^n
- \frac{\Dt}{\eps^2}\nabla\left(\frac{g}{2}h^2\right)^{n+1}
-\frac{\Dt}{\eps^2}gh^{n+1}\nabla b.
\end{equation}
Taking an $L^2$ norm on both sides of \eqref{SWE_s4}, it becomes a quadratic equation for $|\bq^{n+1}|$
\begin{equation}
\frac{\Dt }{\eps^2}\,\frac{gk^2}{(h^{n+1})^{\eta}}|\bq^{n+1}|^2 + |\bq^{n+1}| - |\bq_{\star}| =0.
\label{SWE_s5}
\end{equation}
\eqref{SWE_s5} has two real roots, and the positive value is the reasonable one
\begin{equation*}
\ds |\bq^{n+1}| = \frac{-\eps^2+\sqrt{\eps^4+\ds \frac{4\Dt gk^2}{(h^{n+1})^{\eta}}|\eps^2\bq_{\star}|}}{\ds \frac{2\Dt gk^2}{(h^{n+1})^{\eta}}}.
\end{equation*}
Substituting it into \eqref{SWE_s4}, we update $\bq^{n+1}$ by 
\begin{equation}
\label{qn1}
\bq^{n+1} = \frac{2\eps^2\bq_{\star}}{\eps^2+\sqrt{\eps^4+ \ds \frac{4\Dt gk^2}{(h^{n+1})^{\eta}}|\eps^2\bq_{\star}|}}.
\end{equation}
However, we note that as $\eps\rightarrow 0$, numerically in \eqref{qn1} we face an issue of $\frac{0}{0}$, and in \eqref{qstar} a dividing by $\eps^2$. Both will lead to numerical instabilities. To avoid these two stiff issues, from \eqref{qstar}, we compute $\eps^2\bq_{\star}$ into one by
\begin{equation*}
    \eps^2 \bq_{\star} = \eps^2\bq^n - \eps^2\Dt\nabla \cdot \left(\frac{\bq\otimes \bq}{h}\right)^n- \Dt \nabla\left(\frac{g}{2}h^2\right)^{n+1}-\Dt gh^{n+1}\nabla b.
\end{equation*}
If the solution is at a hydrostatic equilibrium, namely $\nabla\left(\dfrac{1}{2}h^2\right)=-h\nabla b$ and $\bq= \mathbf{0}$,
we would expect $\bq$ always to be $\mathbf{0}$. However, when $\eps$ is small, due to round-off errors, the updating of $\bq^{n+1}$ by \eqref{qn1} 
is unstable. In this case, if $|\eps^2\bq_\star|
<\eta $ and $|\bq^n|<\eta$ where $\eta$ is taken to be $10^{-12}$ at the round-off error level, we set $\bq^{n+1}=\mathbf{0}$.

{\subsubsection{\bf{High order IMEX scheme}} 
The first order IMEX scheme \eqref{SWE_s3} can be generalized to high order with an additive IMEX RK method, such procedure has been introduced in \cite{boscarino2014high}.
For the {additive} system (\ref{Add_IMEX_PDE}), an additive IMEX RK time discretization is characterized by the double Butcher $tableau$ \cite{butcher2016},
\begin{equation}\label{DBT2}
\begin{array}{c|c}
\tilde{\mathbf{c}} & \tilde{A}\\
\hline
\vspace{-0.25cm}
\\
& \tilde{\mathbf{b}^T} \end{array} \ \ \ \ \ \qquad
\begin{array}{c|c}
{\mathbf{c}} & {A}\\
\hline
\vspace{-0.25cm}
\\
& {\mathbf{b}^T} \end{array},
\end{equation}
where $\tilde{A} = (\tilde{a}_{ij})$ is an $s \times s$ lower triangular matrix, i.e. $\tilde{a}_{ij}=0$ for $j \geq i$, leading to an explicit scheme, while $A = ({a}_{ij})$ is an  $s \times s$ matrix for an implicit scheme. A diagonally implicit RK (DIRK) scheme with $a_{ij}=0$ for $j > i$ is used, which is simple and efficient for implicit discretization. Other vectors are $\tilde{\mathbf{c}}=(\tilde{c}_1,...,\tilde{c}_s)^T$, $\tilde{\mathbf{b}}=(\tilde{b}_1,...,\tilde{b}_s)^T$, and $\mathbf{c}=(c_1,...,c_s)^T$, $\mathbf{b}=(b_1,...,b_s)^T$, where $\tilde{\mathbf{c}}$ and $\mathbf{c}$ satisfy the relation
$\tilde{c}_i = \sum_{j=1}^{i-1} \tilde a_{ij}$, $c_i = \sum_{j=1}^{i} a_{ij}$ for $i = 1, ..., s$.

A sketched procedure of the additive IMEX RK scheme is given as follows. Starting from $U^n$, for inner stages $i = 1,\ldots,s$:
\bit
\item Compute the solution $U_{\star}^{(i)}$ from previous known stages
\begin{equation}
\label{ustar}
 U_{\star}^{(i)} = U^n + \Delta t\sum_{j=1}^{i-1}\tilde{a}_{ij}F(U^{(j)})
               + \Delta t\sum_{j=1}^{i-1}a_{ij}G(U^{(j)}).
\end{equation}
\item Update $U^{(i)}$ with
  \begin{equation}
  \label{ui}
  U^{(i)} = U_{\star}^{(i)} + \Delta t a_{ii}G(U^{(i)}).
  \end{equation}
\eit
After obtaining all inner stage values, we update the solution $U^{n+1}$ from
\begin{equation}
  \label{AD_high_final}
U^{n+1} = U^n + \Delta t\sum_{i=1}^{s}\tilde{b}_{i}F(U^{(i)})
+ \Delta t\sum_{i=1}^{s}b_{i}G(U^{(i)}).
\end{equation}

Applying \eqref{ustar}-\eqref{AD_high_final} to \eqref{Add_IMEX_PDE} with \eqref{split-AD}, in components, for $U_{\star}^{(i)}$ in \eqref{ustar}, they are
\begin{subequations}
	\label{ustarc}
\begin{align}
h_{\star}^{(i)} = \, & h^n - \Delta t\sum_{j=1}^{i-1}\tilde{a}_{ij}\left\{\nabla \cdot \bq^{(j)}  + \mu(\eps)\nabla\cdot\left( \sqrt{\frac{h^{\eta+1}}{k^2}}\frac{\nabla H}{\sqrt{|\nabla H|}}\right)^{(j)}
\right\} \\[4mm]
 & + \Delta t\sum_{j=1}^{i-1}a_{ij} \mu(\eps)\nabla\cdot\left( \sqrt{\frac{h^{\eta+1}}{k^2}}\frac{\nabla H}{\sqrt{|\nabla H|}}\right)^{(j)}, \notag
\\[4mm]
\bq^{(i)}_{\star} = \, &  \bq^n  - \Delta t\sum_{j=1}^{i-1}\tilde{a}_{ij}\nabla \cdot \left(\frac{\bq\otimes \bq}{h}\right)^{(j)}- \Delta t\sum_{j=1}^{i-1}a_{ij}\left\{\frac{1}{\eps^2}\nabla\left(\frac{g}{2}h^2\right)  + \frac{1}{\eps^2}gh\nabla b + \frac{1}{\eps^2}\gamma\bq
\right\}^{(j)},
\end{align}
\end{subequations}
and for $U^{(i)}$ in \eqref{ui}, they are
\begin{subequations}
	\label{uic}
\begin{align}
\label{uic_1}
&h^{(i)} = h_{\star}^{(i)} + a_{ii}\Dt \mu(\eps)\nabla\cdot\left( \sqrt{\frac{h^{\eta+1}}{k^2}}\frac{\nabla H}{\sqrt{|\nabla H|}}\right)^{(i)},
\\
\label{uic_2}
&\bq^{(i)} = \bq_{\star}^{(i)} - a_{ii}\frac{\Dt }{\eps^2}\left\{\nabla\left(\frac{g}{2}h^2\right)  + gh\nabla b + \gamma\bq\right\}
^{(i)}.
\end{align}
\end{subequations}
The nonlinear system \eqref{uic_1} is also solved for $h^{(i)}$ by using a simple Picard iteration, which is almost the same as the first order scheme.
The updating $\bq^{(i)}$ is similar to \eqref{qn1}, here it reads
\begin{equation}
  \bq^{(i)} = \frac{2\eps^2\bq_{\star\star}^{(i)}}{\eps^2+\sqrt{\eps^4+ \frac{4a_{ii}\Dt gk^2}{(h^{(i)})^{\eta}}|\eps^2\bq_{\star\star}^{(i)}|}},
\end{equation}
with 
\begin{equation*}
    \eps^2 \bq_{\star\star}^{(i)} = \eps^2 \bq_{\star}^{(i)}  - a_{ii}\Dt \left\{ \nabla \left(\frac{g}{2}h^2\right)  + gh\nabla b \right\}^{(i)}.
\end{equation*}
For the high order IMEX RK scheme, as discussed in \cite{boscarino2019high,boscarino2022high},
 a stability property is needed to ensure that the scheme is both AP and AA.
 Specifically, we require that the IMEX RK scheme is \emph{globally stiffly accurate} (GSA)  with $\tilde{\mathbf{b}}^T=\be_s^T\tilde{A}$ for the explicit part of the scheme and  $\mathbf{b}^T=\be_s^TA$ for  the implicit one where  $\be_s=(0,\cdots,0,1)^T$ \cite{hairer1993solving2}. Moreover, when the IMEX RK scheme is GSA, the numerical solution coincides with the last stage of the method, i.e., $U^{n+1} = U^{(s)}$. From now on, we only consider GSA IMEX RK schemes.

\subsection{High order spatial discretizations}\label{Dspace}
For spatial approximations of \eqref{split-AD}, we follow \cite{huang2022high} by using a finite difference WENO reconstruction for convective terms, while a high order compact difference method for second order diffusion terms. To preserve a hydrostatic equilibrium state, a well-balanced reconstruction \cite{xing2005high} between the gradient of pressure and bottom topography is adopted. Unlike the work \cite{huang2022high}, we do not need to form an elliptic equation for the system \eqref{SWe_MB4} in the diffusive scaling.
 
Without loss of generality, we take the 2D problem of \eqref{split-AD} as an example to describe our spatial approximations. 1D and higher dimensions can be discretized similarly. 
Denoting $U = (h,hu,hv)^T$, the fluxes in \eqref{split-AD} can be written as 
  \begin{align*}
    &F(U) = F_x(U) + F_y(U), \\
    &G(U) = G_x(U) + G_y(U),
  \end{align*}
with 
  \begin{align*}
    &F_x(U) = \left(-(hu)_x - \mu(\eps)\left(\sqrt{\frac{h^{\eta+1}}{k^2}} \frac{H_x}{\sqrt{H_x + H_y}}\right)_x, \,\, -(hu^2)_x,  \,\,-(huv)_x \right)^T, \\
    &F_y(U) = \left(-(hv)_y - \mu(\eps)\left(\sqrt{\frac{h^{\eta+1}}{k^2}} \frac{H_y}{\sqrt{H_x + H_y}}\right)_y, \,\,-(huv)_y,  \,\,-(hv^2)_y\right)^T,\\
    &G_x(U) = \left(\mu(\eps)\left(\sqrt{\frac{h^{\eta+1}}{k^2}} \frac{H_x}{\sqrt{H_x + H_y}} \right)_x, \,\,-\frac{1}{\eps^2}\left(\frac{g}{2} h^2\right)_x - \frac{1}{\eps^2}ghb_x -\frac{1}{\eps^2}\frac{gk^2\sqrt{(hu)^2 + (hv)^2}}{h^{\eta}}(hu), \,\,0 \right)^T,\\
    &G_y(U) = \left(\mu(\eps)\left(\sqrt{\frac{h^{\eta+1}}{k^2}} \frac{H_y}{\sqrt{H_x + H_y}} \right)_y, \,\,0,\,\, -\frac{1}{\eps^2}\left(\frac{g}{2} h^2\right)_y - \frac{1}{\eps^2}ghb_y -\frac{1}{\eps^2}\frac{gk^2\sqrt{(hu)^2 + (hv)^2}}{h^{\eta}}(hu) \right)^T.
  \end{align*}
For simplicity, we consider a rectangular domain $\Omega=[a,b]\times[c,d]$, with a uniform Cartesian partition. Let $\Dx=\frac{b-a}{N_x}$ and $\Dy=\frac{d-c}{N_y}$ be the mesh sizes along the $x$ and $y$ directions respectively, where $N_x$ and $N_y$ are the number of cells correspondingly.
The cell is denoted as $I_{ij}=[x_{i-\frac{1}{2}}, x_{i+\frac{1}{2}}]\times[y_{j-\frac{1}{2}}, y_{j+\frac{1}{2}}]$ with $i=1,2,\ldots,N_x$ and $j=1,2,\ldots,N_y$, where $a=x_{\frac{1}{2}}<x_{\frac{3}{2}}<\cdots<x_{N_x + \frac{1}{2}}=b$ and $c=y_{\frac{1}{2}}<y_{\frac{3}{2}}<\cdots<y_{N_y + \frac{1}{2}}=d$.
Now we set the grid point $(x_i,y_j)$ as $x_{i}=\frac{1}{2}(x_{i-\frac{1}{2}} + x_{i+\frac{1}{2}})$ and $y_j=\frac{1}{2}(y_{j-\frac{1}{2}} + y_{j+\frac{1}{2}})$, located at each cell center. We denote $\omega_{i,j}$ as the value at the point $(x_i,y_j)$, where $\omega$ is $h$ or $\bq$.
Correspondingly $\omega^\pm_{i+\frac{1}{2},j}$, $\omega^\pm_{i,j+\frac{1}{2}}$ represent cell interface values at $(x_{i+\frac12},y_j)$ and $(x_i,y_{j+\frac12})$ from the left and right limits, respectively. A high order finite difference WENO reconstruction will be used for those convective terms, which is done in a dimension-by-dimension way for the 2D problems:
  \begin{itemize}
    \item {\bf{WENO reconstruction with a Lax-Friedrichs flux splitting.}} \\
    Here we briefly review a finite difference WENO reconstruction with a Lax-Friedrichs flux splitting for the convective terms 
    $\nabla\cdot\bq$ and $\nabla\cdot\left(\frac{\bq\otimes\bq}{h}\right)$, we refer to \cite{jiang1996efficient,shu1998essentially} for more details. Taking $\nabla\cdot \bq$ as an example, it is in the form 
  \begin{equation}
    \nabla\cdot\bq = (hu)_x + (hv)_y,
  \end{equation}
  and we approximate them with a $k$th order finite difference reconstruction, namely
  \begin{subequations}
    \begin{align} 
      &(hu)_x\big|_{(x_i,y_j)} = \frac{1}{\Dx} \left\{\widehat{(hu)}_{i+\frac{1}{2},j} - \widehat{(hu)}_{i-\frac{1}{2},j}\right\} + \mathcal{O}(\Dx^k), \\
      &(hv)_y\big|_{(x_i,y_j)} = \frac{1}{\Dy} \left\{\widehat{(hv)}_{i,j+\frac{1}{2}} - \widehat{(hv)}_{i,j-\frac{1}{2}}\right\} + \mathcal{O}(\Dy^k).
    \end{align}
  \end{subequations}
  ${\widehat{(hu)}_{i+\frac{1}{2},j}}$ and $\widehat{(hv)}_{i,j+\frac{1}{2}}$ are obtained by a high order finite difference WENO reconstruction. 
  For example, ${\widehat{(hu)}_{i+\frac{1}{2},j}}$ is decomposed into the following two parts
  \begin{equation} 
    \label{WENO-Flux}
    \widehat{(hu)}_{i+\frac{1}{2},j} = \widehat{(hu)}_{i+\frac{1}{2},j}^{+} + \widehat{(hu)}_{i+\frac{1}{2},j}^{-}.
  \end{equation}
  $\widehat{(hu)}_{i+\frac{1}{2},j}^{+}$ and $\widehat{(hu)}_{i+\frac{1}{2},j}^{-}$ are reconstructed by using upwind and downwind information based on a Lax-Friedrichs flux splitting, with
  \begin{equation}
    \label{LF-Splitting-1}
    (hu)_{i,j}^{\pm} = \frac{1}{2}\left\{(hu)_{i,j} \pm \Lambda H\right\}, \quad \Lambda = \max\left\{|\bu| + \min\{1,\frac{1}{\eps^2}c\}\right\}.
  \end{equation}
  For example, for a fifth order WENO reconstruction, the upwind flux $\widehat{(hu)}_{i+\frac{1}{2},j}^{+}$ is obtained from $\{ (hu)_{i-2,j}^+, (hu)_{i-1,j}^+, (hu)_{i,j}^+,(hu)_{i+1,j}^+,(hu)_{i+2,j}^+ \}$, while the downwind flux $\widehat{(hu)}_{i+\frac{1}{2},j}^{-}$ is obtained based on $\{ (hu)_{i-1,j}^-, (hu)_{i,j}^-, (hu)_{i+1,j}^-,(hu)_{i+2,j}^-,(hu)_{i+3,j}^- \}$, respectively.
  Note that here we take the surface level $H$ in \eqref{LF-Splitting-1} for numerical viscosity, which in the hydrostatic state \eqref{Still-water} will not introduce any numerical dissipation, so that can help to preserve such an equilibrium steady state.
  $(hv)_{i,j}^{\pm}$ can be obtained similarly, and it is also similar for the flux difference approximation of $\nabla\cdot\left(\frac{\bq\otimes\bq}{h}\right)$. We denote such a WENO reconstruction as $\nabla_{LW}$.

  \item  {\bf{Well-balanced WENO reconstrction.}} \\
  If initially the water is at a hydrostatic state \eqref{Still-water}, such an equilibrium will be preserved for later times. Numerically it is also very important to keep such an equilibrium state. Many well-balanced methods have been proposed \cite{kurganov2002central,zhou2001surface,leveque1998balancing,xing2005high,X2017}. Here we follow \cite{xing2005high} to obtain a well-balanced WENO reconstruction. For the system, in order to preserve $\bq=\mathbf{0}$ when $H$ is a constant, it is important to preserve
\begin{equation*}
	\nabla \left(\frac{g}{2}h^2\right) = -gh\nabla b.
\end{equation*}
The main idea is to split the right side as
\begin{equation*}
	\nabla \left(\frac{g}{2}h^2\right) = -gh \nabla b=-g H\nabla b +\nabla\left(\frac{g}{2}b^2\right),
\end{equation*}
and apply the same WENO reconstruction (with exactly the same nonlinear weights) of $\nabla \left(\frac{g}{2}h^2\right)$ to $\nabla b$ and $\nabla\left(\frac{g}{2}b^2\right)$, so that
\begin{equation*}
	\nabla \left(\frac{g}{2}h^2\right) + g H\nabla b - \nabla\left(\frac{g}{2}b^2\right) = \nabla\left(\frac{g}{2} H^2 \right)
\end{equation*}
also holds at the discrete level, when $H$ is a constant.
Such a WENO reconstruction is denoted as $\nabla_{W}^{WB}$.
  
  \item {\bf{Compact central difference discretization for diffusive terms}} \\
  For the diffusive terms in \eqref{split-AD}, a compact central difference discretization will be applied \cite{boscarino2019high,boscarino2022high,huang2022high}. For example, 
  a fourth order central difference for $(a(x,y)H_x)_x$ at  $(x_i,y_j)$  takes the following form

\begin{equation}
	\label{CD4}
	\begin{aligned}
		(a(x,y)H_x)_x\Big|_{(x_i, y_j)}
		=\,&\frac{1}{\Delta x^2}
		{\bf a}_{i,j}
		\begin{pmatrix}
			-25/144 & 1/3 & -1/4 & 1/9 & -1/48\\
			1/6   & 5/9 &  -1  & 1/3 & -1/18\\
			0    &  0  &   0  &  0  &   0  \\
			-1/18  & 1/3 &  -1  & 5/9 &  1/6 \\
			-1/48  & 1/9 & -1/4 & 1/3 &-25/144
		\end{pmatrix}
		{\bf H}^T_{i,j} \\[3mm]
		&+\mathcal{O} (\Delta x^4),
	\end{aligned}
\end{equation}
with the two vectors being
\begin{align}
	{\bf a}_{i,j}&=(a_{i-2,j},a_{i-1,j},a_{i,j},a_{i+1,j},a_{i+2,j}), \notag\\[3mm]
	{\bf H}_{i,j}&=(H_{i-2,j},H_{i-1,j},H_{i,j},H_{i+1,j},H_{i+2,j}), \notag
\end{align}
similarly for $(a(x,y)H_y)_y$ along the $y$ direction. The function $a(x,y)$ represents the coefficient $\sqrt{\frac{h^{\eta +1}}{k^2}}\frac{1}{\sqrt{|\nabla H|}}$. For the gradient $|\nabla H|$ in the coefficient of the diffusive terms, as no upwind information is available, so a WENO reconstruction without numerical viscosity (namely $\Lambda=0$ as in \eqref{LF-Splitting-1}) is used, which we denote as $|\nabla_W H|$.
We note that as we observed numerically a WENO reconstruction of $|\nabla H|$ is important to control numerical oscillations, especially when the solution is not smooth.
   
\end{itemize} 

\subsection{\bf{Algorithm flowchart}}
Now we will summarize our proposed scheme, with high order finite difference WENO recontructions and compact central difference discretizations, coupled with high order IMEX RK schemes, which is AP and well-balanced for the SWEs with a nonflat bottom topology and Manning friction \eqref{SWe_MB4}.
The algorithm flowchart is as follows.
\begin{enumerate}
    \item  Start from the time level $t^n$, with $U^n$, we first compute $U^{(i)}_{\star}$ from (\ref{ustar}):
    \begin{subequations}
      \begin{align}
      \label{U_sta_IMEX3_E1}
      h_{\star}^{(i)} = \, & h^n - \Delta t\sum_{j=1}^{i-1}\tilde{a}_{ij}\left\{\nabla_{LW} \cdot \bq^{(j)}  + \mu(\eps)\nabla\cdot\left( \sqrt{\frac{h^{\eta+1}}{k^2}}\frac{\nabla H}{\sqrt{|\nabla_{W} H|}}\right)^{(j)}
      \right\} \\
      &+ \Delta t\sum_{j=1}^{i-1}a_{ij} \mu(\eps)\nabla \cdot\left( \sqrt{\frac{h^{\eta+1}}{k^2}}\frac{\nabla H}{\sqrt{|\nabla_{W} H|}}\right)^{(j)}, \notag
      \\
      \bq^{(i)}_{\star} = \, & \bq^n  - \Delta t\sum_{j=1}^{i-1}\tilde{a}_{ij}\nabla_{LW} \cdot \left(\frac{\bq\otimes \bq}{h}\right)^{(j)} - \Delta t\sum_{j=1}^{i-1}a_{ij}\left\{\frac{1}{\eps^2}\left[\nabla_{W}^{WB}\left(\frac{g}{2}h^2\right)  \right. \right. \\
      &\left. \left. + g(h+b)\nabla_W^{WB} b - \nabla_{W}^{WB}\left(\frac{g}{2}b^2\right)\right]+ \frac{1}{\eps^2}\frac{gk^2|\bq|}{h^{\eta}}\bq\right\}^{(j)}. \notag
      \end{align}
      \end{subequations}
      \item Then we solve $U_I^{(i)}$ from \eqref{ui}:
      \begin{enumerate}
        \item $U_I^{(i)}$ can be expressed as
      \begin{subequations}
        \label{ui_space}
        \begin{align}
          h^{(i)} = \,&h_{\star}^{(i)} + a_{ii}\mu(\eps)\Dt\nabla\cdot\left(\sqrt{\frac{h^{\eta + 1}}{k^2}} \frac{\nabla H}{\sqrt{|\nabla_W H|}}\right)^{(i)},\label{ui_space-1}\\
          \bq^{(i)} = \, &\bq_{\star}^{(i)}  - \Dt a_{ii}\left\{\frac{1}{\eps^2}\left[\nabla_{W}^{WB}\left(\frac{g}{2}h^2\right) + g(h+b)\nabla_W^{WB} b  \right. \right. \\ 
           & \left.\left.  - \nabla_{W}^{WB}\left(\frac{g}{2}b^2\right)\right]+ \frac{1}{\eps^2}\frac{gk^2|\bq|}{(h^{\eta})}\bq\right\}^{(i)}. \notag
        \end{align}
      \end{subequations}
      \item For the system \eqref{ui_space},  we first obtain $h^{(i)}$ from \eqref{ui_space-1} by using a Picard iteration as the first order IMEX scheme, that is, we set $h^{(i),0} = h^{(i-1)}$ and compute $|\nabla_W H^{(i),0}|$ with a WENO reconstruction $\nabla_W$. If it is less than $\delta$, we set $h^{(i)}=h^{(i),0}$ and stop the iteration, otherwise, we continue the iteration by 
      \begin{equation}
        \label{hi_picard}
        h^{(i),\ell+1} = h_{\star}^{(i)} - a_{ii}\mu(\eps)\Dt\nabla\cdot\left\{\left(\frac{\sqrt{\frac{h^{\eta +1}}{k^2}}}{\sqrt{|\nabla_W H|}}\right)^{(i),\ell}\nabla H^{(i),\ell+1}\right\},\quad \ell \ge 0.
      \end{equation}
      The elliptic operator $\nabla\cdot\left\{\left(\frac{\sqrt{\frac{h^{\eta +1}}{k^2}}}{\sqrt{|\nabla_W H|}}\right)^{(i),\ell}\nabla H^{(i),\ell+1}\right\}$ in \eqref{hi_picard} is discreted by a high order central difference discretization as described in \eqref{CD4}.
      The iteration is stoped when $||h^{(i),\ell+1} - h^{(i),\ell}|| \le \delta$.
      
      \item After obtaining $h^{(i)}$, we update $\bq^{(i)}$ from
      \begin{equation}
        \bq^{(i)} = \frac{2\eps^2\bq_{\star\star}^{(i)}}{\eps^2+\sqrt{\eps^4+ \ds\frac{4a_{ii}\Dt gk^2}{(h^{(i)})^{\eta}}|\eps^2\bq_{\star\star}^{(i)}|}},
      \end{equation}
      with 
      \begin{equation}
        \eps^2 \bq_{\star\star}^{(i)} = \eps^2 \bq_{\star}^{(i)}  - a_{ii}\Dt \left\{ \nabla_{W}^{WB}\left(\frac{g}{2}h^2\right)  + g(h+b)\nabla_{W}^{WB} b - \nabla_{W}^{WB}\left(\frac{g}{2}b^2\right) \right\}^{(i)}.
      \end{equation}
      Similarly, when $|\eps^2 \bq_{\star\star}^{(i)}|$ and $|\bq^{(i-1)}|$ both are less than $\delta$, we set $\bq^{(i)} = \mathbf{0}$. 
    \end{enumerate}
    \item Finally, we update the numerical solution $U^{n+1} = U^{(s)}$ with  the assumption of IMEX RK schemes being GSA.
  \end{enumerate}

\section{Asymptotic preserving (AP) and Asymptotically Accurate (AA)}
\label{sec4}
\setcounter{equation}{0}
\setcounter{figure}{0}

In this section, we will formally show that the first order IMEX scheme \eqref{SWE_s3} is AP, and the high order IMEX RK scheme \eqref{ustar}-\eqref{AD_high_final} is AA.
For both properties, we only consider time discretizations, while keeping space continuous. In accordance with \eqref{Exp1}, we adopt the following Chapman-Enskog expansions:
\begin{equation}
  \label{hqn}
  h^n(\bx) = h^n_0(\bx) + \mathcal{O}(\eps), \qquad
  \bq^n(\bx) = \bq^n_0(\bx) + \mathcal{O}(\eps),
\end{equation}
with $H^n_0(\bx) = h^n_0(\bx) + b(\bx)$, and
\begin{equation}
	\label{hq_relation}
	\bq^n_0 = - \left(\sqrt{\frac{h_0^{\eta+1}}{k^2}}\frac{\nabla H_0}{\sqrt{|\nabla H_0|}}\right)^{n}.
\end{equation}
Here $h^n(\bx) = h(\bx,t^n), \, \bq^n(\bx) = \bq(\bx,t^n)$, and we say that the initial conditions are {\em well-prepared} if the condition \eqref{hqn} and \eqref{hq_relation} are satisfied at $n=0$.

\subsection{AP property.}
For the first order IMEX scheme \eqref{SWE_s3}, we have
\begin{thm}
As $\eps \rightarrow 0$ the first order IMEX scheme \eqref{SWE_s3} is asymptotic preserving as long as the initial conditions are well-prepared.
\end{thm}
\begin{proof}
  Such a proof is straightforward by using the mathematical induction. 
  As $\eps$ goes to zero, the system \eqref{SWE_s3} becomes
  \begin{subequations}
    \label{AP_E1}
    \begin{align}
      & \frac{h_0^{n+1} - h_0^{n}}{\Dt} + \nabla \cdot \bq^n_0 + \mu(0)\nabla \cdot \left(\sqrt{\frac{h_0^{\eta+1}}{k^2}}\frac{\nabla H_0}{\sqrt{|\nabla H_0|}}\right)^n = \mu(0)\nabla \cdot \left(\sqrt{\frac{h_0^{\eta+1}}{k^2}}\frac{\nabla H_0}{\sqrt{|\nabla H_0|}}\right)^{n+1}, \label{AP_E1-1} \\
      & \nabla \left(\frac{g}{2}h_0^2\right)^{n+1} = -gh_0^{n+1}\nabla b - \frac{gk^2|\bq_0^{n+1}|}{h_0^{\eta}}\bq_0^{n+1}. \label{AP_E1-2}
    \end{align}
  \end{subequations}
  From \eqref{AP_E1-2}, following \eqref{eq2.4}-\eqref{lim_E3}, we obtain 
  \begin{equation}
  	\label{qn+1}
    \bq_0^{n+1} = -\left(\sqrt{\frac{h_0^{\eta +1}}{k^2}} \frac{\nabla H_0}{\sqrt{|\nabla H_0|}}\right)^{n+1}.
  \end{equation}
  Substituting the relation of \eqref{qn+1} at time level $t^n$ into \eqref{AP_E1-1}, since $\mu(0) = 1$, as $\varepsilon \to 0$ we get
  \begin{equation}
  	\label{hn+1}
    \frac{h_0^{n+1} - h_0^{n}}{\Dt} = \nabla \cdot \left(\sqrt{\frac{h_0^{\eta+1}}{k^2}}\frac{\nabla H_0}{\sqrt{|\nabla H_0|}}\right)^{n+1},
  \end{equation}
 which is a consistent implicit time discretization to the limiting equation \eqref{lim_E4-1}. Starting from well-prepared initial conditions \eqref{hqn} at $n=0$, by the mathematical induction, we can see the AP property holds for all later time $t^n$. 
\end{proof}

\subsection{AA property.}
Similarly as the AP property for the first order IMEX scheme \eqref{SWE_s3}, here we focus on the AA analysis for high order IMEX RK schemes \eqref{ustar}-\eqref{AD_high_final}. In what follows,  we recognize that the GSA condition is crucial to guarantee the AA property.
\begin{thm}
Consider applying an IMEX RK scheme  \eqref{ustar}-\eqref{AD_high_final} of order $p$. Assume that the IMEX RK method is GSA and the the initial conditions are well-prepared. Let us denote by $\bU^1(\bx;\eps)=(h^1(\bx;\eps),\bq^1(\bx;\eps))^T$ the numerical solution after one time step and  $\bU^{exa}(\bx,t)=(h^{exa}(\bx,t),\bq^{exa}(\bx,t))^T$ the exact solution of the limiting equation \eqref{lim_E4-1}. Then after one step, the error between the numerical and exact solution satisfies the following estimate
  \begin{equation}
    \label{esti_error}
    \lim_{\eps\to 0} \,\bU^1(\bx;\eps) = \bU^{exc}(\bx,\Dt) + \mathcal{O}(\Dt^{p+1}),
  \end{equation}
  namely, the scheme is AA.
\end{thm}

\begin{proof}
  We prove it by using the mathematical induction for time stages. If we assume \eqref{hqn} and \eqref{hq_relation} hold at $t^n$, from the AP property of the first order IMEX scheme \eqref{SWE_s3}, we can get \eqref{qn+1} and \eqref{hn+1} hold for a time step $a_{11}\Delta t$ at the first stage. Then if we assume this formal convergence holds for the $(i-1)$-th stage, we will prove it also holds for the $i$-th stage. We have:
  \begin{enumerate}
    \item  for $U_{\star}^{(i)}$ in \eqref{ustarc}, as $\eps\rightarrow 0$, it becomes
    \begin{align*}
          h_{\star}^{(i)} = \,&  h_0^n - \Delta t\sum_{j=1}^{i-1}\tilde{a}_{ij}\left\{\nabla \cdot \bq_0^{(j)}  + \mu(0)\nabla\cdot\left( \sqrt{\frac{h_0^{\eta+1}}{k^2}}\frac{\nabla H_0}{\sqrt{|\nabla H_0|}}\right)^{(j)}
          \right\}, \\
          & + \Delta t\sum_{j=1}^{i-1}a_{ij} \mu(0)\nabla\cdot\left( \sqrt{\frac{h_0^{\eta+1}}{k^2}}\frac{\nabla H_0}{\sqrt{|\nabla H_0|}}\right)^{(j)}, \\
          \bq^{(i)}_{\star} = \,&
           \bq_0^n  - \Delta t\sum_{j=1}^{i-1}\tilde{a}_{ij}\nabla \cdot \left(\frac{\bq_0\otimes \bq_0}{h_0}\right)^{(j)}.
    \end{align*}
    Since 
   \begin{equation*}
 	\label{qj}
 	\bq_0^{(j)} = -\left(\sqrt{\frac{h_0^{\eta +1}}{k^2}} \frac{\nabla H_0}{\sqrt{|\nabla H_0|}}\right)^{(j)}, \quad 0\le j \le i-1,
   \end{equation*}
   and $\mu(0) = 1$ as $\varepsilon \to 0$,
   $U^{(i)}_{\star}$ becomes
    \begin{align*}
	h_{\star}^{(i)} = \,&  h_0^n + \Delta t\sum_{j=1}^{i-1}a_{ij} \nabla\cdot\left( \sqrt{\frac{h_0^{\eta+1}}{k^2}}\frac{\nabla H_0}{\sqrt{|\nabla H_0|}}\right)^{(j)}, \\
	\bq^{(i)}_{\star} = \,&
	\bq_0^n  - \Delta t\sum_{j=1}^{i-1}\tilde{a}_{ij}\nabla \cdot \left(\frac{\bq_0\otimes \bq_0}{h_0}\right)^{(j)}.
    \end{align*}   
    \item for $U^{(i)}$ in \eqref{uic}, as $\eps\rightarrow 0$, first from \eqref{uic_2} we get
      \begin{equation*}
        \nabla\left(\frac{g}{2}(h_0^{(i)})^2\right)  + gh_0^{(i)}\nabla b + \gamma^{(i)}_0\bq_0^{(i)} = 0,
      \end{equation*}
          with $\gamma^{(i)}_0 = \dfrac{gk^2|\bq_0^{(i)}|}{(h_0^{\eta})^{(i)}}$, so that 
          \begin{equation*}
            \bq_0^{(i)} = - \left(\sqrt{\frac{h_0^{\eta +1}}{k^2}}\frac{\nabla H_0}{\sqrt{|\nabla H_0|}} \right)^{(i)},
          \end{equation*}
          and \eqref{uic_1} becomes 
          \begin{equation*}
              h_0^{(i)} = h_{\star}^{(i)} + \Dt a_{ii}\nabla \cdot\left(\sqrt{\frac{h_0^{\eta+1}}{k^2}}\frac{\nabla H_0}{\sqrt{|\nabla H_0|}}\right)^{(i)}=h_0^n + \Dt \sum_{j=1}^{i} a_{ij}\nabla \cdot\left(\sqrt{\frac{h_0^{\eta+1}}{k^2}}\frac{\nabla H_0}{\sqrt{|\nabla H_0|}}\right)^{(j)}.
          \end{equation*}
          Therefore, the formal convergence also holds for the stage $i$.
  \end{enumerate}
With the inner stage $i$ up to $s$, and given that the IMEX RK scheme is GSA, we update the solution $U^{n+1}$ as  $U^{n+1} = U^{(s)}$. Therefore, as $\eps\rightarrow 0$, the scheme becomes a $p$th order implicit scheme for the limiting equation \eqref{lim_E4-1} using the matrix $A$ in \eqref{DBT2}. This means that the error after one time step satisfies \eqref{esti_error}, confirming that the scheme is AA.
\end{proof}

\section{Numerical tests}
\label{sec5}
\setcounter{equation}{0}
\setcounter{figure}{0}
\setcounter{table}{0}

In this section, we will perform some numerical experiments for the SWEs with Manning friction and a non-flat bottom topology, with our proposed scheme. A third order GSA IMEX RK scheme from \cite{ascher1997implicit} is adopted, with the double Butcher {\it tableau} given by
\begin{align*} \label{IMEX1_(5,5,3)}
&\begin{array}{c|ccccc|ccccc}
0 & 0 & 0 & 0 & 0 & 0 &0&  0 & 0  & 0 & 0\\[3mm]
\frac{1}{2} & \frac{1}{2} & 0 & 0 & 0 & 0& 0&  \frac{1}{2} & 0 & 0& 0\\[3mm]
\frac{2}{3} & \frac{11}{18}&\frac{1}{18} & 0 & 0&0& 0&  \frac{1}{6} & \frac{1}{2} & 0&0\\[3mm]
\frac{1}{2} &   \frac{5}{6} & -\frac{5}{6} &  \frac{1}{2}& 0&0&0 & -\frac{1}{2}& \frac{1}{2} & \frac{1}{2}&0\\[3mm]
1 & \frac{1}{4}&\frac{7}{4} & \frac{3}{4} & -\frac{7}{4}&0&0 & \frac{3}{2}& -\frac{3}{2} & \frac{1}{2}&\frac{1}{2}\\[3mm]
\hline
  & \frac{1}{4}&\frac{7}{4} & \frac{3}{4} & -\frac{7}{4}&0 &0 & \frac{3}{2}& -\frac{3}{2} & \frac{1}{2}&\frac{1}{2}
\end{array}.
\end{align*}
In space, a fifth order finite difference WENO scheme \cite{jiang1996efficient,shu1998essentially}, and a fourth order compact central difference discretization, introduced in section \ref{Dspace}, are used to reconstruct the first and second order spatial derivatives, respectively. Overall the resulting scheme is third order in time and fourth order in space, and denoted as ``T3S4''. For comparison, we also consider a first order scheme,
with a first order IMEX scheme \eqref{SWE_s3} in time and a first order 
numerical flux for \eqref{WENO-Flux}, which are 
\begin{equation}
  \widehat{(hu)}^{+}_{i+\frac{1}{2},j} = (hu)^{+}_{i,j}, \quad 
  \widehat{(hu)}^{-}_{i+\frac{1}{2},j} = (hu)^{-}_{i+1.j},
\end{equation}
while a second order central difference scheme \cite{boscarino2019high} is used for second order derivatives. The scheme is denoted as ``T1S1''.
Here the time step is taken to be $\Delta t = \text{CFL}\,\Delta x/\Lambda$ with  $\Lambda = \max \{|\bu|+\min(1,1/\eps)\sqrt{gh}\}$
and $\text{CFL}=0.2$. Numerical results will validate that our proposed scheme is high order in space and in time, AP and AA in the diffusive limit, as well as well-balancing for capturing propagation of small perturbations around a hydrostatic state.

\begin{exa}{
\em
\label{exam57}({\bf{1D accuracy test with linear friction}})
In this example, we start to test the order of convergence for the T3S4 scheme, with different values of $\eps$ in the one-dimensional case. For simplicity, we consider the system \eqref{SWe_MB4} with a linear friction term where $\gamma=1$ is taken. We consider a well-prepared initial condition, in order to show both AP and AA properties, which is given as follows:
\begin{equation}
h(x,0) = \sin(\pi x) + 2,  \qquad q(x,0) = -2\pi\cos(\pi x)\left(\sin(\pi x) + 2 \right),
\end{equation}
with a flat bottom $b(x) = 0$. The computational domain is $[0,2]$ with a periodic boundary condition, and $g=2$.
We take three different values of $\eps$: $\eps=1,10^{-2},10^{-6}$.
Setting the final time $T=0.02$, we compute the numerical errors and order of convergence, by comparing to a reference solution with $N=2560$.
The numerical errors and orders are shown in Table~ \ref{T_exam57_1}.
From the results, we can see that when $\eps=1$ and $\eps =10^{-6}$, with mesh refinement, a 3rd order accuracy has been obtained, namely the temporal accuracy dominates. For the intermediate regime $\eps=10^{-2}$, order reductions can be observed, which is a typical case occurred in such high order IMEX methods \cite{boscarino2019high,boscarino2022high,huang2022high}.
From this table, we verify that the AP and AA properties are obtained.

\begin{table}[htbp]
  \caption{ Example \ref{exam57}. $L^1$ errors and orders for $h$ and $hu$ with different $\eps$'s. $T=0.02$. }
  \begin{center}
		\begin{tabular}{c|c|c|c|c|c|c|c}
			\hline\hline	
		\multicolumn{1}{c}{\multirow{2}*{}}&\multicolumn{1}{|c|}{\multirow{2}*{N}}&\multicolumn{2}{c|}{ $1$}&\multicolumn{2}{c|}{$10^{-2}$}&\multicolumn{2}{c}{$10^{-6}$}\\
			\cline{3-8}
\multicolumn{1}{c}{} &\multicolumn{1}{|c|}{} &error& order& error&order&error&order\\  \hline\hline
\multicolumn{1}{c|}{\multirow{6}*{$h$}}
    &40 &     9.30E-04 &       --&     4.88E-06 &      -- &     6.39E-06 &       --\\ \cline{2-8}
    &80 &     7.26E-05 &     3.68&     2.73E-06 &     0.84&     5.35E-07 &     3.58\\ \cline{2-8}
   &160 &     2.18E-06 &     5.06&     9.20E-07 &     1.57&     5.89E-08 &     3.18\\ \cline{2-8}
   &320 &     1.10E-07 &     4.31&     2.39E-07 &     1.94&     8.34E-09 &     2.82\\ \cline{2-8}
   &640 &     7.02E-09 &     3.97&     4.92E-08 &     2.28&     1.11E-09 &     2.91\\ \cline{2-8}
   &1280&     5.62E-10 &     3.64&     7.34E-09 &     2.75&     1.26E-10 &     3.14\\ \hline\hline
\multicolumn{1}{c|}{\multirow{6}*{$hu$}}
    &40 &     6.35E-03 &       --&     1.84E-04 &       --&     1.02E-04 &     --\\ \cline{2-8}
    &80 &     4.64E-04 &     3.77&     5.24E-05 &     1.81&     8.49E-06 &     3.58\\ \cline{2-8}
   &160 &     1.39E-05 &     5.06&     1.48E-05 &     1.83&     1.36E-06 &     2.65\\ \cline{2-8}
   &320 &     6.78E-07 &     4.36&     4.87E-06 &     1.60&     2.02E-07 &     2.75\\ \cline{2-8}
   &640 &     4.44E-08 &     3.93&     1.20E-06 &     2.02&     2.70E-08 &     2.90\\ \cline{2-8}
   &1280&     3.65E-09 &     3.61&     2.01E-07 &     2.57&     3.11E-09 &     3.11\\ \hline\hline
\end{tabular}
\end{center}
\label{T_exam57_1}
\end{table}

}
\end{exa}

\begin{exa}
  {\em
  \label{exam12}
  ({\bf 1D accuracy test with nonlinear friction}) 
  In this example, we will test the order of convergence for our scheme with  nonlinear friction. Following \cite{yang2021high}, we construct an exact solution as follows: 
  \begin{equation}
    h(x,t)=2 + \eps^2\sin(\pi (x-t)),\qquad q(x,t) =2 + \eps^2\sin(\pi (x-t)).
    \label{exam12_F1}
  \end{equation}
  This solution satisfies the SWEs with extra source terms
  \begin{equation}
  	\label{exam12_eq}
  \left\{
  \begin{array}{ll}
  h_t + q_x = 0,\\ [3mm]
 \ds q_t +\left(\frac{q^2}{h}\right)_x + \frac{1}{\eps^2}\left(\frac{g}{2}h^2\right)_x = -\frac{1}{\eps^2}gk^2\frac{|q|q}{h^{7/3}} + \frac{1}{\eps^2}g k^2\left[ 2 + \eps^2\sin(\pi(x-t)) \right]^{-1/3} \\[3mm] \hspace{4.4cm} +  g \pi h\cos( \pi(x-t)),
  \end{array}
  \right.
  \end{equation}
  Here, a flat bottom $b(x) = 0$ is considerred and the Manning coefficient is $k=1$, and $g=1$. The computational domain is $\Omega =[0,2]$.
  For this example, due to those extra source terms in \eqref{exam12_eq}, we are not able to take $\eps$ approaching $0$. Instead we
  set $\eps=1$ and take a final time $T=0.04$. We compute the numerical errors and orders by comparing to the exact solutions, and show the results in Table~\ref{T_exam12_1}.
  From the table, we observe almost fifth-order accuracy. This might be due to that spatial errors of convective terms are dominated in this hyperbolic regime. 
  
\begin{table}[htbp]
  \caption{ Example \ref{exam12}. $L^1$ errors and orders for $h$ and $hu$ with $\eps=1$. $T=0.04$.}
  \begin{center}
    \begin{tabular}{c|c|c|c|c}
      \hline\hline	
      \multicolumn{1}{c|}{\multirow{2}*{N}}&\multicolumn{2}{c|}{ $h$}&\multicolumn{2}{c}{$hu$}\\
      \cline{2-5}
\multicolumn{1}{c|}{} &error& order& error&order\\  \hline\hline
    10 &     2.37E-03 &       --&     5.23E-03 &       --\\ \cline{1-5}
    20 &     8.11E-05 &     4.87&     1.51E-04 &     5.11\\ \cline{1-5}
    40 &     2.46E-06 &     5.04&     3.04E-06 &     5.64\\ \cline{1-5}
    80 &     7.59E-08 &     5.02&     7.89E-08 &     5.27\\ \cline{1-5}
   160 &     2.33E-09 &     5.02&     2.05E-09 &     5.27\\ \cline{1-5}
   320 &     6.93E-11 &     5.07&     4.13E-11 &     5.63\\ \cline{1-5}
   640 &     1.96E-12 &     5.14&     4.72E-12 &     3.13\\ \cline{1-5}
   1280&     1.71E-13 &     3.52&     7.68E-13 &     2.62\\\hline\hline
\end{tabular}
\end{center}
\label{T_exam12_1}
\end{table}
  
  }\end{exa}

  \begin{exa}{
    \em
    \label{exam55}
    In this example, we will verify the AP property of our scheme. We consider two initial conditions: smooth and discontinuous, as in \cite{bulteau2020fully}. The spatial domain for both cases is $[-5,5]$.
    The smooth initial condition is given by
    \begin{equation}
    h(x,0) = \left\{
                \begin{aligned}
            &2  ,   & x< -1;\\
                    &\frac{1}{2}\left(3+ \sin\left(\frac{3\pi x}{2}\right)\right), & -1\le x < 1;\\
            &1,   & \text{otherwise};
          \end{aligned}
          \right.
    \qquad q(x,0) = 0,
    \label{exam55_F1}
    \end{equation}
    while the discontinuous initial condition is:
    \begin{equation}
    h(x,0) = \left\{
                \begin{aligned}
            &2  ,   & x < 0;\\
            &1,   & \text{otherwise};
          \end{aligned}
          \right.
    \qquad q(x,0) = 0.
    \label{exam55_F2}
    \end{equation}
    Since the stiffness of \eqref{SWe_MB4} is caused by the long time and strong friction, to investigate this, we consider two tests: 
    (1) we fix the final time $T=0.01$ and set the coefficient of friction $g k^2=1.0$, where $g=9.812$. We take four different values of $\eps$,  $1,\,0.2,\,0.1,\,0.0005$, and compare the numerical solutions with the  solution obtained from the limiting equation \eqref{lim_E4-1}, which we denote as ``lim''. The results obtained by T3S4 are shown in Fig.~\ref{Fig_ex55_1}, with $N=200$.
    From the figure, we observe that as $\eps$ goes to zero, the solutions approach to the solution for the limiting equation, namely, our scheme preserve the AP property.
    (2) we fix $\eps=1$ (namely without dimensionalization), but take a final time $T=0.01\theta$ where $\theta = \sqrt{gk^2}$, to test the long-time behavior of the solutions.
    We take four different values of $\theta=1,\,5,\,10,\,50$. The results of water depth $h$ are shown in Fig.~\ref{Fig_ex55_2} with the same grid $N=200$.
    It can be seen that the numerical results are almost the same as the first case, which confirm the scaling in \eqref{Dim_kt}.
    
    Finally, we compare the results obtained by the high order T3S4 scheme with those obtained by the first order T1S1 scheme. We take two cases.
    In the first case, the coefficient of the friction is $gk^2=1$ and the final time is $T=0.01$, where the parameter $\eps=0.0005$. In the second case, we choose $gk^2=2500$, $T=0.5$, and $\eps=1$. The results are shown in Fig.~\ref{Fig_ex55_3} and Fig.~\ref{Fig_ex55_4}, respectively.
    From the results, we observe that as $\eps$ goes to zero, or the simulation time increases, the first order solution only match the high order solution under a very refined mesh size,  which show the necessity of using high order schemes. 
    
    \begin{figure}[hbtp]
      \begin{center}
      \mbox{
          \subfigure[water depth $h$]
      {\includegraphics[width=7.0cm]{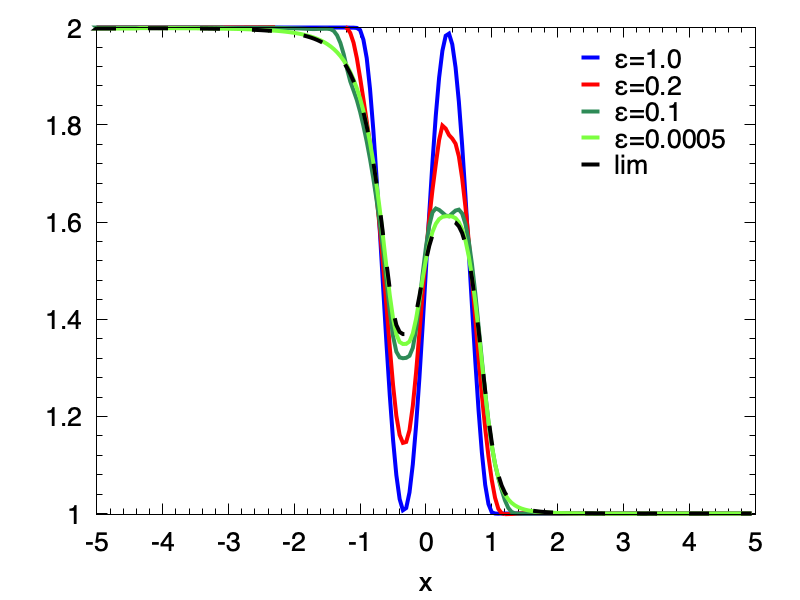}}\quad
          \subfigure[momentum $q$]
      {\includegraphics[width=7.0cm]{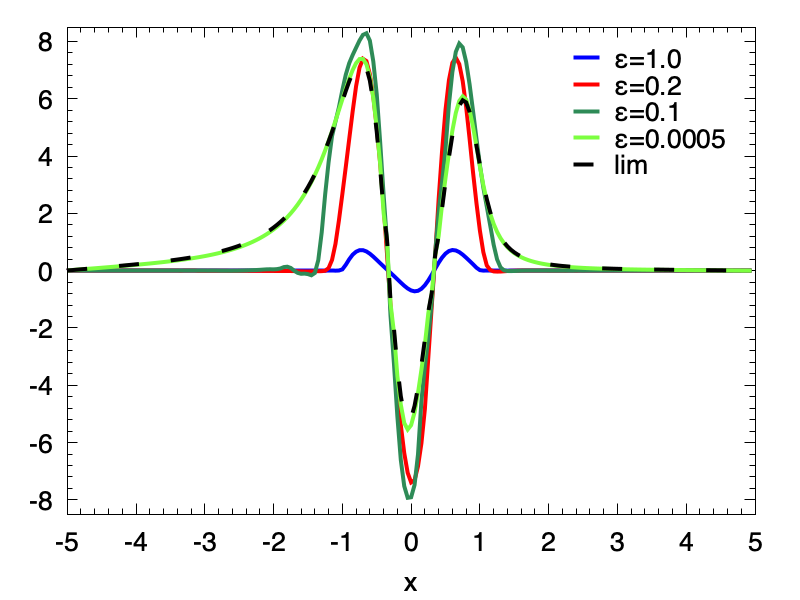}}\quad
      }
      \mbox{
          \subfigure[water depth $h$]
      {\includegraphics[width=7.0cm]{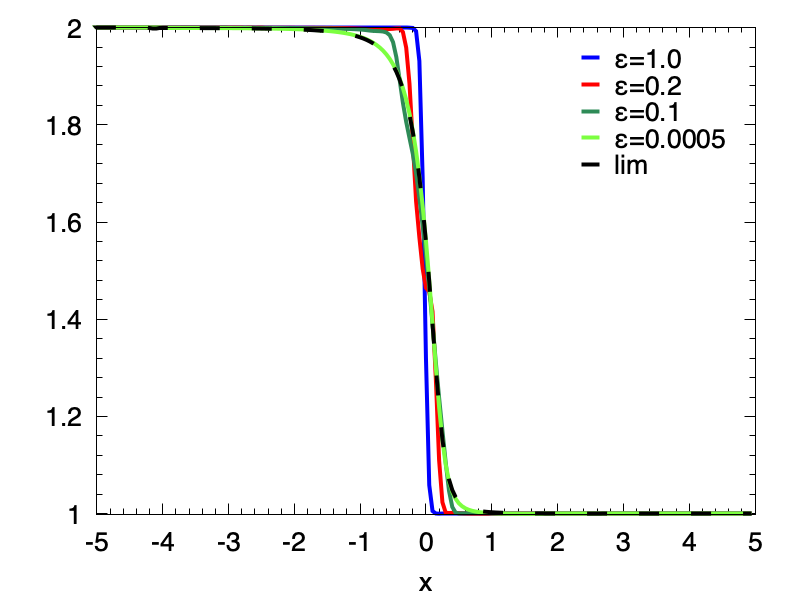}}\quad
          \subfigure[momentum $q$]
      {\includegraphics[width=7.0cm]{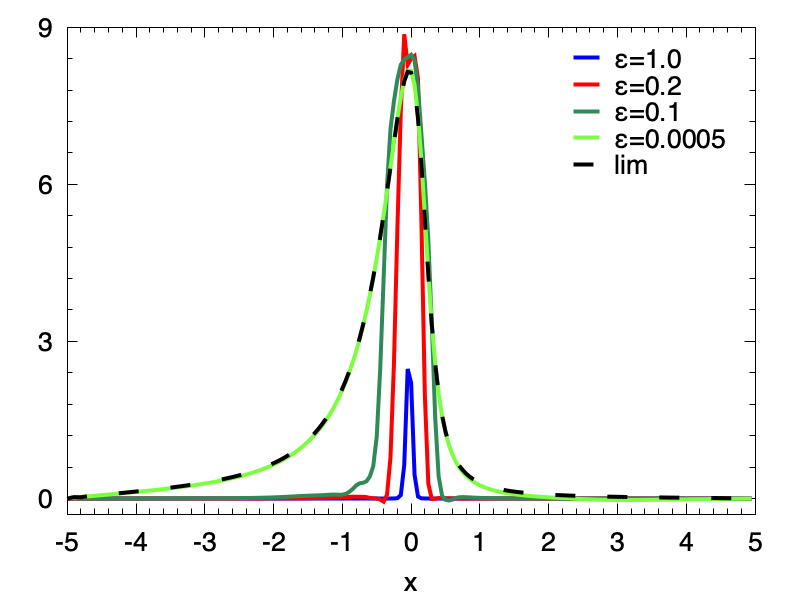}}\quad
      }
      \caption{ Example~\ref{exam55}. Numerical solutions of the water depth $h$ (left) and the momentum $q$ (right) for the smooth initial condition \eqref{exam55_F1} (top) and the discontinuous initial condition \eqref{exam55_F2} (bottom) with different $\eps$'s .}
      \label{Fig_ex55_1}
      \end{center}
      \end{figure}

      \begin{figure}[hbtp]
        \begin{center}
        \mbox{
            \subfigure[water depth $h$]
        {\includegraphics[width=7.0cm]{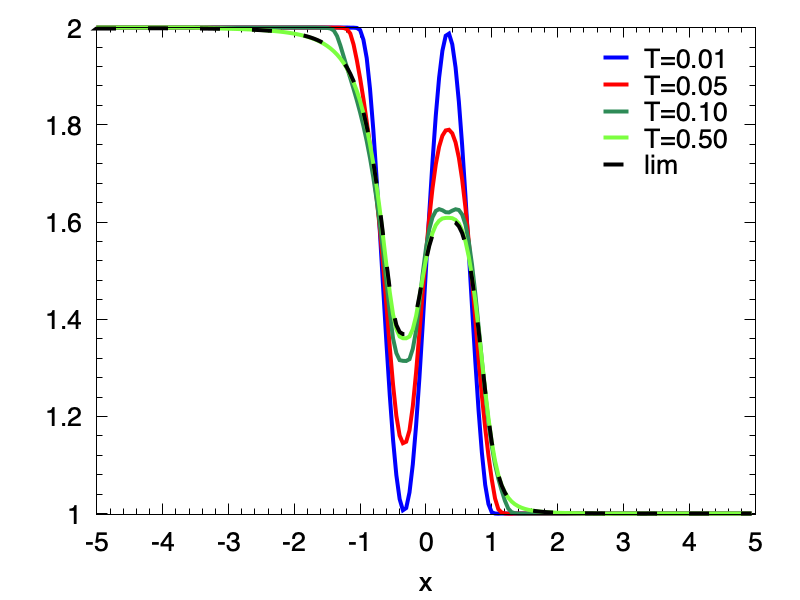}}\quad
            \subfigure[water depth $h$]
        {\includegraphics[width=7.0cm]{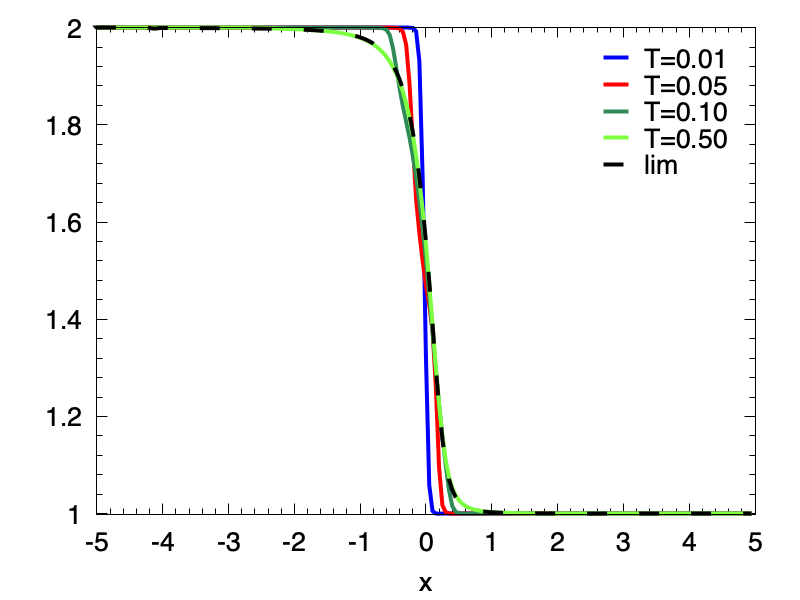}}\quad
        }
        \caption{ Example~\ref{exam55}. Numerical solutions of the water depth $h$ for the smooth initial condition \eqref{exam55_F1} (left) and the discontinuous initial condition \eqref{exam55_F2} (right) at different time. }
        \label{Fig_ex55_2}
        \end{center}
        \end{figure}
     
        \begin{figure}[hbtp]
          \begin{center}
          \mbox{
              \subfigure[$\eps=0.0005$, $gk^2=1$, $T=0.01$]
          {\includegraphics[width=7.0cm]{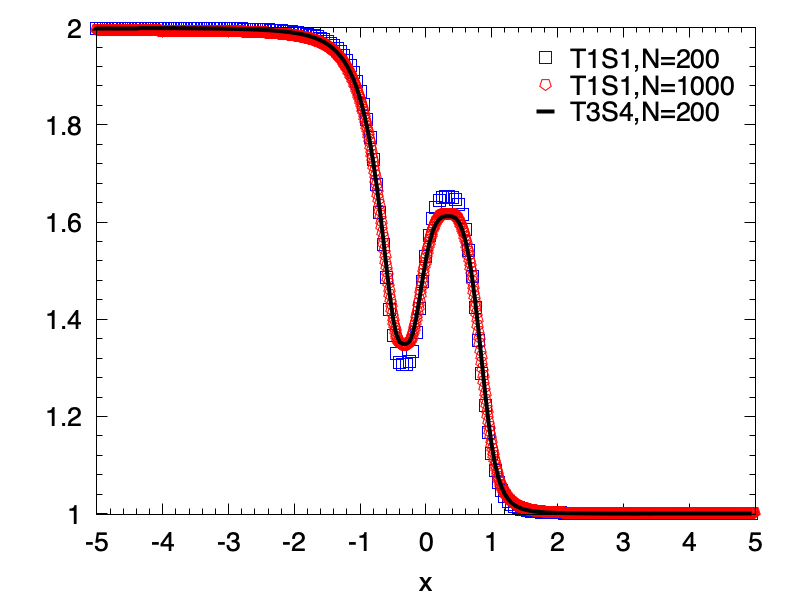}}\quad
              \subfigure[$\eps=0.0005$, $gk^2=1$, $T=0.01$]
          {\includegraphics[width=7.0cm]{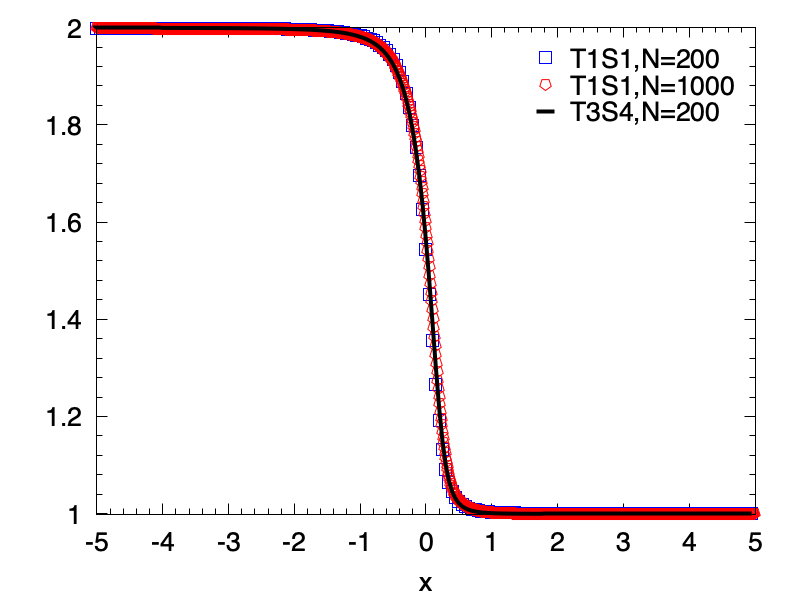}}\quad
          }
          \caption{ Example~\ref{exam55}. Numerical solutions of the water depth $h$ for the smooth initial condition \eqref{exam55_F1} (left) and the discontinuous initial condition \eqref{exam55_F2} (right), with the T1S1 and T3S4 schemes.}
          \label{Fig_ex55_3}
          \end{center}
          \end{figure}
    
          \begin{figure}[hbtp]
            \begin{center}
            \mbox{
                \subfigure[$\eps=1$, $gk^2=2500$, $T=0.5$]
            {\includegraphics[width=7.0cm]{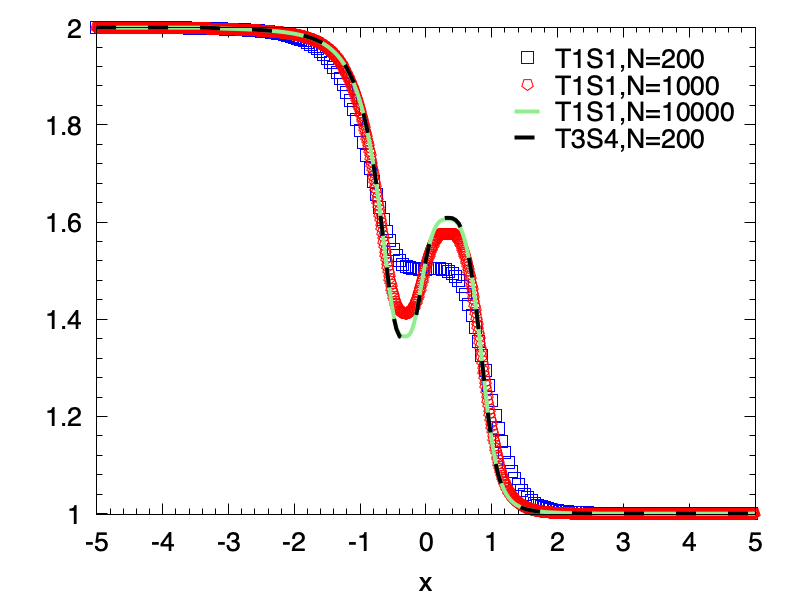}}\quad
                \subfigure[$\eps=1$, $gk^2=2500$, $T=0.5$]
            {\includegraphics[width=7.0cm]{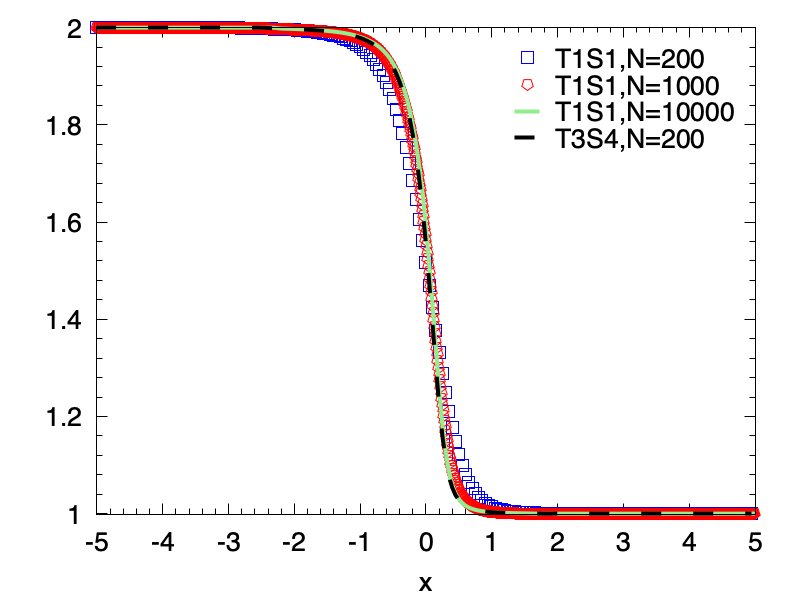}}\quad
            }
            \caption{ Example~\ref{exam55}. Numerical solutions of the water depth $h$ for the smooth initial conditions \eqref{exam55_F1} (left) and the discontinuous initial condition \eqref{exam55_F2} (right), with the T1S1 and T3S4 schemes.}
            \label{Fig_ex55_4}
            \end{center}
            \end{figure}
}\end{exa}

\begin{exa}{
  \em
  \label{exam24_V2}
  In this example, we now test the ability of our proposed scheme to capture the propagation of small perturbations over strong friction. The initial condition is taken to be:   
  \begin{equation}
    h(x,0) = \left\{ 
      \begin{aligned}
        &0.5 + 0.1\sin(10\pi(x-1.1)), & 1.1 \le  x \le 1.2;\\
        &0.5, & \text{otherwise};
      \end{aligned}
    \right.
    \quad q(x,0) = 0.
  \end{equation}
  The computational domain is $[0,2]$, with a gravitational constant  $g=9.812$ and a flat bottom $b(x) = 0$. We set the friction coefficient be $k=1$, corresponding to strong friction. Inflow and outflow boundary conditions are considerred. We compute the solution up to a final time $T=0.02$. Five different values of $\eps$ are taken, i.e $1,\,0.5,\,0.1,\,10^{-2},\,10^{-6}$. As $\eps$ becoming smaller, the friction becomes stronger. We compare the results with the one obtained by the limiting equation \eqref{lim_E4-1}, which is denoted as ``lim'',
  and present the numerical results in Fig.~\ref{Fig1_exam24_V2}.
  From the results, we observe that as $\eps$ decreases, the diffusion of the system becomes to dominate. Our scheme can be uniformly stable and capture the small perturbations well. Besides, when $\eps =10^{-6}$, the solution matches very well with the one obtained by the limiting equation \eqref{lim_E4-1}, namely our scheme is AP in the diffusive regime.
  \begin{figure}[hbtp]
    \begin{center}
    \mbox{
        \subfigure[the surface level $H$]
    {\includegraphics[width=7.0cm]{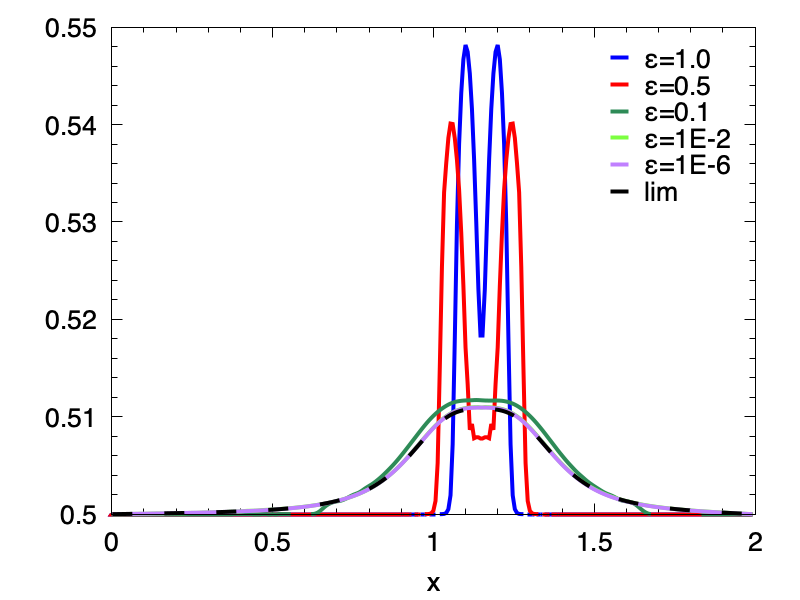}}\quad
        \subfigure[momentum $q$]
    {\includegraphics[width=7.0cm]{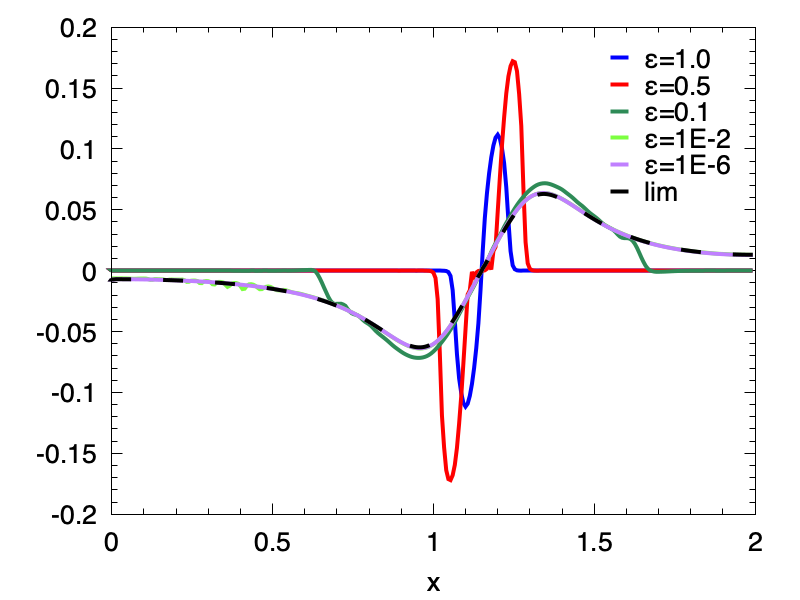}}\quad
    }
    \caption{ Example~\ref{exam24_V2}. Numerical solutions of the surface level $H$ (left) and the momentum $q$ (right), with strong friction.}
    \label{Fig1_exam24_V2}
    \end{center}
  \end{figure}
   
  }\end{exa}

  \begin{exa}{
    \em
    \label{exam24}
    In this example, we will check the well-balanced property of our scheme. 
    An equilibrium initial condition is taken, with  a small perturbation adding to a surface level, following the study in \cite{yang2021high}.
    This setup is an extension of the work presented in \cite{xing2005high, xing2006high, huang2022high}, where the system is considered without friction.

    Here we consider a non-flat trigonometric bottom topology, which is given by
    \begin{equation}
    \label{ex6_1}
    b(x) =
    \left\{
                \begin{aligned}
            &0.25(\cos(10\pi(x-1.5))+1)  ,   & 1.4 \le x\le 1.6;\\
            &0,   & \text{otherwise};
          \end{aligned}
          \right.
    \end{equation}
    The initial condition with a small disturbance is taken:
    \begin{equation}
    \label{ex6_2}
    h(x,0) + b(x) =
    \left\{
                \begin{aligned}
            &1.001  ,   & 1.1 \le x\le 1.2;\\
            &1,   & \text{otherwise};
          \end{aligned}
          \right.
    \qquad q(x,0) = 0.
    \end{equation}
    Here, the computational domain is $[0,2]$, $g=9.812$ and $k=1$. The final time $T=0.2$.
    
    To demonstrate that our scheme is not affected by the dimensionless parameter $\eps$, we scale the initial condition into a dimensionless form, using $\hat{t} = \eps t$, $\hat{k} = \eps k$, and $\hat{q} = \frac{q}{\eps}$, as shown in \eqref{Dim_kt} and \eqref{Dim_q}. 
    We take $N=200$, and consider an inflow and outflow boundary condition. Two different values of $\eps$ are used, i.e., $1$ and $0.6$. 
    After we obtain the numerical results for different $\eps$'s, we transform them back to the dimension case. The results are shown in Fig. ~\ref{Fig_ex24_1}. We can see the results match each other very well, namely, they are not affected by the scaling. Besides, these results are also the same as in \cite{yang2021high}, which can capture the small perturbations well. This demonstrates that our scheme is well-balanced independent of the parameter $\eps$. 
    Finally, we compare the results obtained by the high order T3S4 scheme and the first order T1S1 scheme, which are shown in Fig.~\ref{Fig_ex24_2}. Clearly the high order results are much better than the first order results, and they match well when the mesh is fine enough for the first order scheme.
    
    \begin{figure}[hbtp]
    \begin{center}
    \mbox{
    {\includegraphics[width=8.0cm]{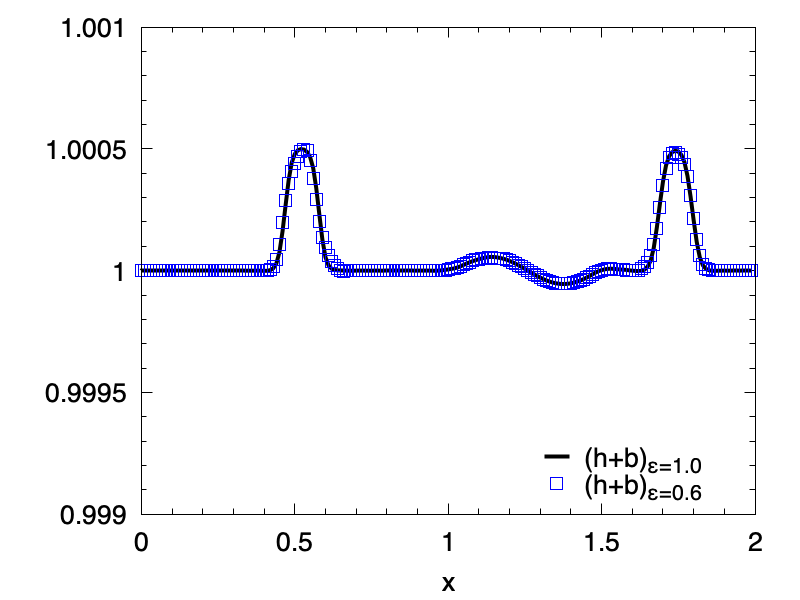}}\quad
    {\includegraphics[width=8.0cm]{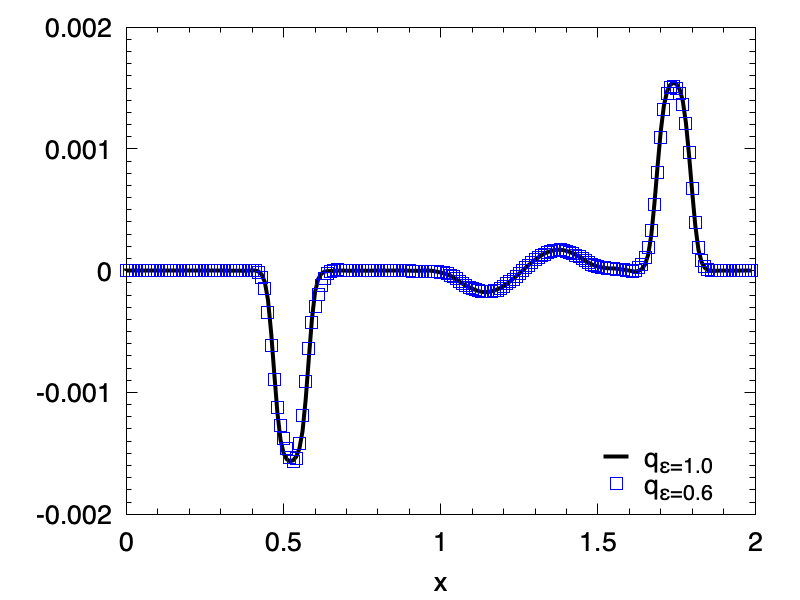}}\quad
    }
    \caption{ Example~\ref{exam24}. Numerical solutions of the surface level $H$ (left) and the momentum $q$ (right) for the T3S4 scheme with different $\eps$'s.}
    \label{Fig_ex24_1}
    \end{center}
    \end{figure}
    
    \begin{figure}[hbtp]
      \begin{center}
      \mbox{
      {\includegraphics[width=8.0cm]{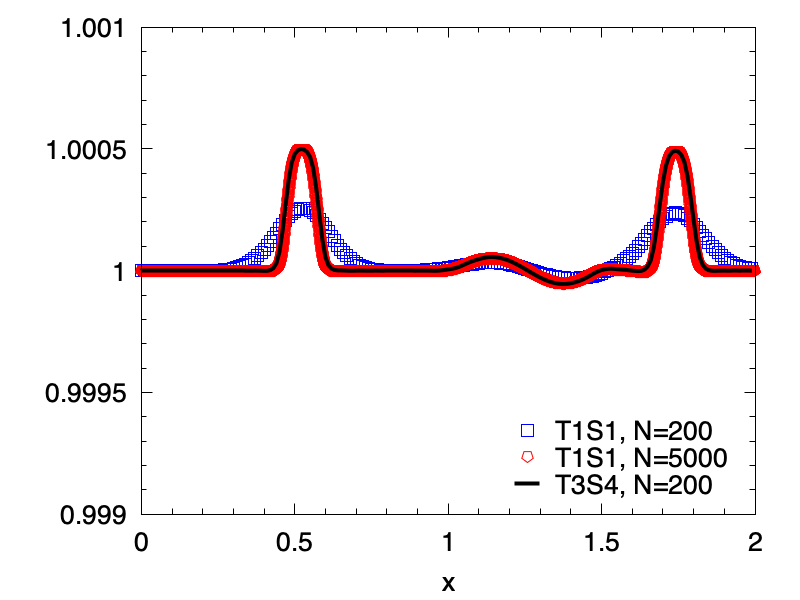}}\quad
      {\includegraphics[width=8.0cm]{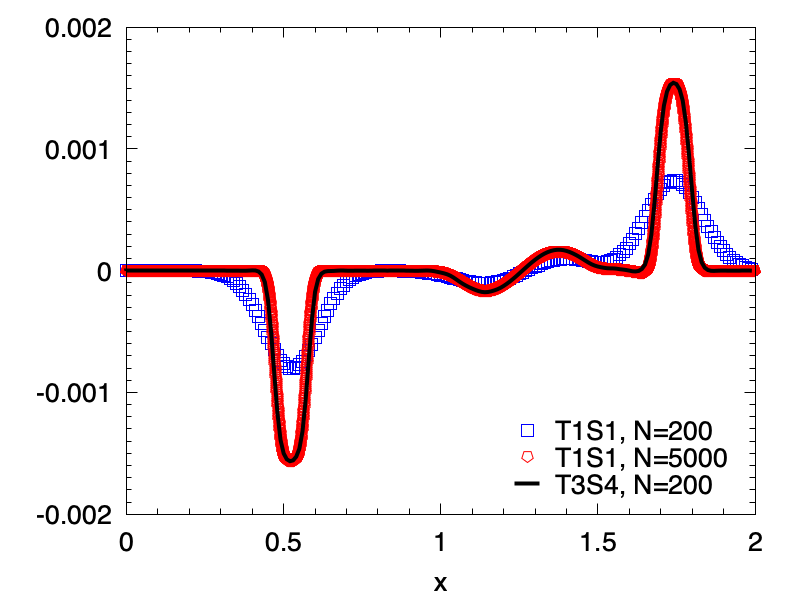}}\quad
      }
      \caption{ Example~\ref{exam24}. Numerical solutions of the surface level $H$ (left) and the momentum $q$ (right), with the T1S1 and T3S4 schemes. $\eps=1$.}
      \label{Fig_ex24_2}
      \end{center}
      \end{figure}
    
    }\end{exa}

    \begin{exa}{
      \em\label{exam42}
      In this example, we consider a shock wave with a non-flat bottom \cite{yang2021high}. The Manning coefficient is $k = 0.09$, $g=9.812$. The non-flat bottom has two humps which is defined as follows:
      \begin{equation}
      b(x) =
      \left\{
            \begin{aligned}
              &\max(0,0.25-20(x -0.25)^2),   & 0 \le x \le 0.5;\\
              &\max(0,0.25-20(x -0.75)^2),   & 0.5 < x \le 1;
            \end{aligned}
            \right.
      \end{equation}
      The initial condition is given by
      \begin{equation}
      h(x,0)+ b(x) =
      \left\{
                  \begin{aligned}
              &1  ,   & 0 \le x \le 0.5;\\
              &0.5,   & 0.5 < x \le 1;
            \end{aligned}
            \right.
            \quad
            q(x,0) = 0.
      \end{equation}
      We take a computational domain $[0,1]$ with $N=200$ and set $\eps=1$.
      An inflow and outflow boundary condition is adopted, with a final time $T=0.08$. The numerical results are shown in Fig.~\ref{Fig_ex42_1}. 
      From the results, we observe that two waves are generated from the initial discontinuity and propagate to left and right, respectively. 
      We also compare the results obtained by the high order T3S4 scheme, with those obtained by the first order T1S1 scheme. Similarly the high order scheme has better resolution than the first order scheme. This example demonstrates the good performance of our high order scheme for shock capturing.
      
      \begin{figure}[hbtp]
      \begin{center}
      \mbox{\subfigure[water surface $h+b$]
      {\includegraphics[width=8.0cm]{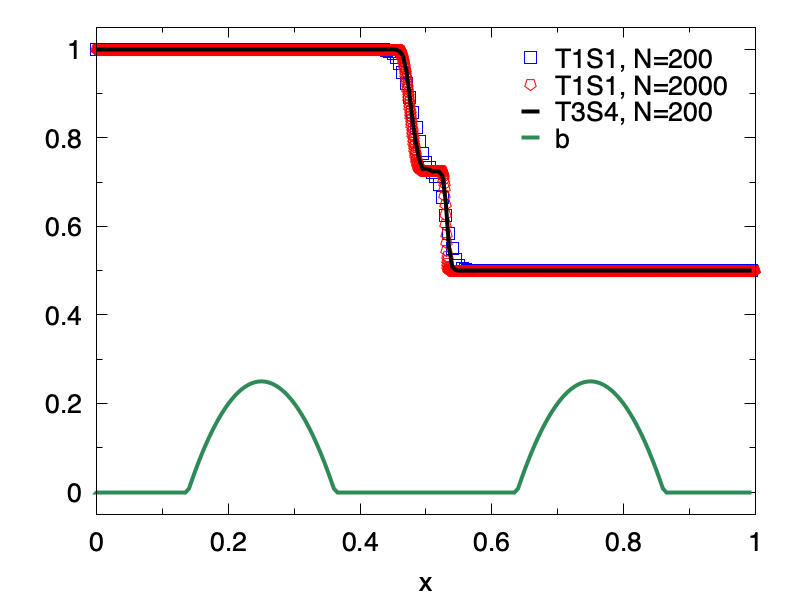}}\quad
      \subfigure[momentum $q$]
      {\includegraphics[width=8.0cm]{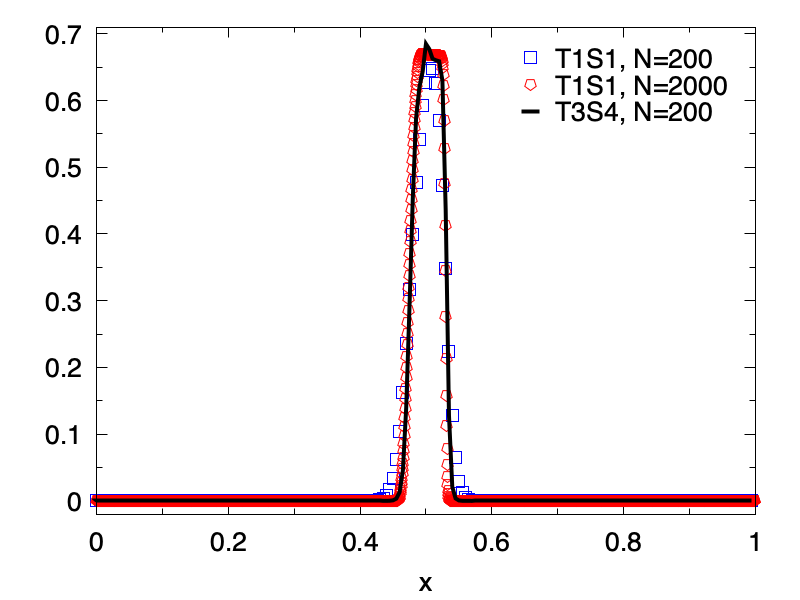}}\quad
      }
      \caption{ Example~\ref{exam42}. Numerical solutions of the surface level $H$ (left) and the momentum $q$ (right). $\eps=1$. $T=0.08$.}
      \label{Fig_ex42_1}
      \end{center}
      \end{figure}
    
      
      }\end{exa}
      
      \begin{exa}{
        \em
        \label{2D_linear_test}({\bf{2D accuracy test with linear friction}})
        In this example, we test the orders of accuracy for our scheme in the 2D case. Similarly we consider the system \eqref{SWe_MB4} with a linear friction of $\gamma=1$. 
        The initial conditions are given by:
        \begin{equation}
          \left\{
          \begin{aligned}
            &   h(x,y,0) = \sin(\pi (x+y)) + 2, \\
            & (hu)(x,y,0) = -2\pi\cos(\pi (x+y))\left(\sin(\pi (x+y)) + 2 \right),\\
            & (hv)(x,y,0) = -2\pi\cos(\pi (x+y))\left(\sin(\pi (x+y)) + 2 \right).
          \end{aligned}
          \right.
      \end{equation}
      The bottom is flat with $b(x,y) = 0$.
      The computational domain is $\Omega = [0,2]^2$, and we set $g=2$. Periodic boundary conditions are used. The computational domain is divided uniformly into $N_\ell^2$ meshes, where $N_\ell = N_0\cdot 2^k$ with $N_0=8$ and $\ell=0,1,2,3,4,5$. 
      Since the exact solution is not available, we take the numerical solution obtained by $N^2=512^2$ as the reference solution, and compute the numerical errors by comparing the numerical solution and the reference solution. Here we consider three different values of $\eps$: $1,10^{-2},10^{-6}$.
       
      The numerical errors and orders are shown in Table~ \ref{T_2D_linear_test_1} at a final time $T=0.01$.
      From the results, we observe that around fourth order accuracy for $\eps=1$, and third order for $\eps=10^{-6}$ are obtained, which is similar to the one-dimensional case.
      This demonstrates that our scheme performs well in both hyperbolic and diffusive regimes. 
      An order reduction also appears for the intermediate regime of $\eps=10^{-2}$, similar to the accuracy tests in \cite{boscarino2019high, boscarino2022high, huang2022high}.
        \begin{table}[htbp]
          \caption{ Example \ref{2D_linear_test}. $L^1$ errors and orders  for $h$, $hu$, and $hv$ with $\eps=1,10^{-2},10^{-6}$.}
          \begin{center}
            \begin{tabular}{c|c|c|c|c|c|c|c}
              \hline\hline	
            \multicolumn{1}{c}{\multirow{2}*{}}&\multicolumn{1}{|c|}{\multirow{2}*{N}}&\multicolumn{2}{c|}{ $1$}&\multicolumn{2}{c|}{$10^{-2}$}&\multicolumn{2}{c}{$10^{-6}$}\\
              \cline{3-8}
            \multicolumn{1}{c}{} &\multicolumn{1}{|c|}{} &error& order& error&order& error&order\\  \hline\hline
            \multicolumn{1}{c|}{\multirow{5}*{$h$}}
        & 16 &     5.94E-03 &       --&     2.57E-04 &       --&     2.97E-04 &       --  \\ \cline{2-8}
        & 32 &     5.62E-04 &     3.40&     1.25E-05 &     4.36&     1.78E-05 &     4.06  \\ \cline{2-8}
        & 64 &     2.37E-05 &     4.57&     4.50E-06 &     1.48&     1.37E-06 &     3.71\\ \cline{2-8}
        &128 &     1.31E-06 &     4.17&     1.40E-06 &     1.69&     1.32E-07 &     3.37 \\ \cline{2-8}
        &256 &     8.14E-08 &     4.01&     2.78E-07 &     2.33&     1.39E-08 &     3.25\\ \hline\hline
        \multicolumn{1}{c|}{\multirow{5}*{$hu$}}
        & 16 &     3.82E-02 &       --&     5.16E-03 &       --&     5.52E-03 &       -- \\ \cline{2-8}
        & 32 &     3.72E-03 &     3.36&     3.14E-04 &     4.04&     2.94E-04 &     4.23\\ \cline{2-8}
        & 64 &     1.67E-04 &     4.48&     9.44E-05 &     1.73&     1.80E-05 &     4.03 \\ \cline{2-8}
        &128 &     9.43E-06 &     4.15&     3.00E-05 &     1.65&     2.47E-06 &     2.86 \\ \cline{2-8}
        &256 &     5.52E-07 &     4.10&     7.29E-06 &     2.04&     3.31E-07 &     2.90  \\ \hline\hline
        \multicolumn{1}{c|}{\multirow{5}*{$hv$}}
        & 16 &     3.82E-02 &       --&     5.16E-03 &       --&     5.52E-03 &       --\\ \cline{2-8}
        & 32 &     3.72E-03 &     3.36&     3.14E-04 &     4.04&     2.94E-04 &     4.23\\ \cline{2-8}
        & 64 &     1.67E-04 &     4.48&     9.44E-05 &     1.73&     1.80E-05 &     4.03\\ \cline{2-8}
        &128 &     9.43E-06 &     4.15&     3.00E-05 &     1.65&     2.47E-06 &     2.86\\ \cline{2-8}
        &256 &     5.52E-07 &     4.10&     7.29E-06 &     2.04&     3.31E-07 &     2.90\\ \hline\hline
      \end{tabular}
      \end{center}
      \label{T_2D_linear_test_1}
      \end{table}
        }
        \end{exa}

      \begin{exa}
        {\em
        \label{exam5}
        In this example, we now test the accuracy of the scheme for  nonlinear friction in the two-dimensional case.
        Similar to the one-dimensional problem, the ``manufactured'' exact solutions are taken to be:
        \begin{subequations}
          \begin{align*}
          &h(x,y,t)    = 2 + \eps^2\sin(\pi (x+y-2t)),   \\
          &q_1(x,y,t) = 2 + \eps^2\sin(\pi (x+y-2t)),   \\
          &q_2(x,y,t) = 2 + \eps^2\sin(\pi (x+y-2t)),
          \end{align*}
          \end{subequations}
         which satisfy the SWEs with extra source terms
        \begin{equation}
        \left\{
        \begin{array}{ll}
        \partial_t h + \partial_x q_1 + \partial_y q_2 = 0,\\[3mm]
        \partial_t q_1+ \partial_x \left(\frac{q_1^2}{h}\right)  + \partial_y\left(\frac{q_1q_2}{h} \right)+ \frac{1}{\eps^2}\partial_x\left(\frac{g}{2}h^2\right) = -\frac{1}{\eps^2}\frac{gk^2\sqrt{q_1^2 +q_2^2}}{h^{7/3}}q_1 + \frac{gk^2}{\eps^2}\left[ 2 +  \right.\\[3mm]  \hspace{3.2cm} \left.+\eps^2\sin(\pi(x + y - 2t)) \right]^{-1/3} +  g \pi h \cos(\pi(x + y - 2t)), \\[3mm]
        \partial_t q_2+ \partial_x \left(\frac{q_1q_2}{h}\right)  + \partial_y\left(\frac{q_2^2}{h} \right)+ \frac{1}{\eps^2}\partial_y\left(\frac{g}{2}h^2\right) = -\frac{1}{\eps^2}\frac{gk^2\sqrt{q_1^2 +q_2^2}}{h^{7/3}}q_2 + \frac{gk^2}{\eps^2}\left[ 2 +  \right.\\[3mm]  \hspace{3.2cm} \left.+\eps^2\sin(\pi(x + y - 2t)) \right]^{-1/3} +  g \pi h \cos(\pi(x + y - 2t)).
        \end{array}
        \right.
        \end{equation}
        $g=1$ and $k=1$. The computational domain is $[0,2]^2$, and the final time is $T=0.04$. We only consider $\eps=1$ as the one-dimensional case. The errors and orders are shown in Table~\ref{T_exam5_1}.
        From the table, we observe a desired fifth order accuracy.
        \begin{table}[htbp]
          \caption{ Example \ref{exam5}. The $L^1$ errors and orders  for $h$, $hu$, and $hv$ with $\eps=1$.}
          \begin{center}
            \begin{tabular}{c|c|c|c|c|c|c}
              \hline\hline	
              \multicolumn{1}{c|}{\multirow{2}*{N}}&\multicolumn{2}{c|}{ $h$}&\multicolumn{2}{c|}{$hu$}&\multicolumn{2}{c}{$hv$}\\
            \cline{2-7}
            \multicolumn{1}{c|}{} &error& order& error&order& error&order\\  \hline\hline
          8 &     3.20E-02 &       --&     6.46E-02 &       --&     6.46E-02 &       --\\ \cline{1-7}
         16 &     1.74E-03 &     4.20&     3.62E-03 &     4.16&     3.62E-03 &     4.16\\ \cline{1-7}
         32 &     6.36E-05 &     4.78&     1.03E-04 &     5.13&     1.03E-04 &     5.13 \\ \cline{1-7}
         64 &     2.00E-06 &     4.99&     3.12E-06 &     5.05&     3.12E-06 &     5.05\\ \cline{1-7}
        128 &     6.23E-08 &     5.01&     9.62E-08 &     5.02&     9.62E-08 &     5.02 \\ \cline{1-7}
        256 &     1.93E-09 &     5.01&     2.93E-09 &     5.04&     2.93E-09 &     5.04\\ \hline\hline
      \end{tabular}
      \end{center}
      \label{T_exam5_1}
      \end{table}
          
      }\end{exa}
      
      \begin{exa}
      {\em
      \label{exam2D_2}
      In the final example, we check the well-balanced property for our scheme in the two-dimensional case.
      We consider a bottom topography being an isolated elliptical shaped hump 
          \begin{equation}
          \label{Ex8_1}
          b(x,y) = 0.8e^{-5(x-0.9)^2-50(y-0.5)^2},
          \end{equation}
      and the initial conditions are
      \begin{subequations}
          \begin{equation}
          \label{Ex8_2}
          h(x,y,0)+b(x,y) = \left\{
          \begin{aligned}
          & 1+0.01,  & 0.05  \le x \le 0.15; \\
          & 1,       &\text{otherwise};
          \end{aligned}
          \right.
          \end{equation}
          \begin{equation}
          \label{Ex8_3}
          hu = hv = 0.
          \end{equation}
      \end{subequations}
      This example has been studied in \cite{yang2021high}.
      The computational domain is $[0,2]\times[0,1]$, 
      with an inflow and outflow boundary condition in the $x$ direction and a periodic boundary condition in the $y$ direction.
      Similar experiments were conducted in prior works such as \cite{xing2005high, xing2006high, huang2022high},  without friction. 
      In this case, a small friction  has been added with  Manning coefficient  $k=0.09$. 
      The numerical solutions of the surface level $H=h+b$ are shown in Fig.~\ref{Fig_ex2D_2_1} with two different meshes, i.e., $200\times100$ and $400\times200$. 
      The initial perturbation is separated into two waves propagating to left and right. 
      With the left-propagating wave moving out of the domain, the right-propagating wave interacts with the non-flat bottom topography and is well captured by our proposed T3S4 scheme. 
      The numerical results are similar to those in \cite{yang2021high}.
      \begin{figure}[hbtp]
          \begin{center}
              \mbox{ \subfigure[$t=0.12$]					
                   {\includegraphics[width=7cm]{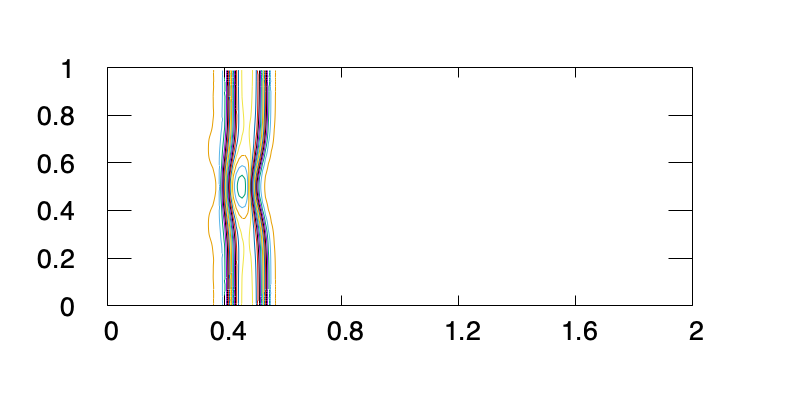}}\quad		
                     \subfigure[$t=0.12$]		
               {\includegraphics[width=7cm]{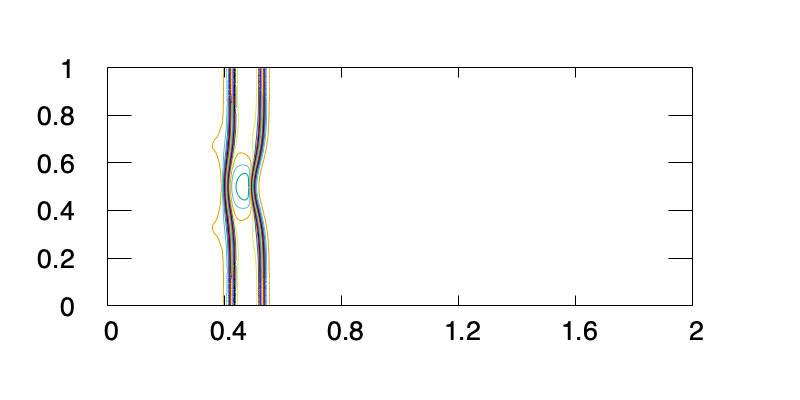}}
              }
              \mbox{ \subfigure[$t=0.24$]					
                   {\includegraphics[width=7cm]{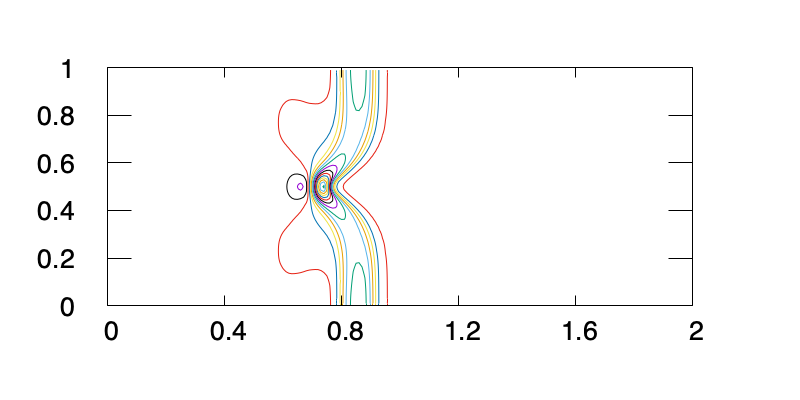}}\quad		
                     \subfigure[$t=0.24$]		
             {\includegraphics[width=7cm]{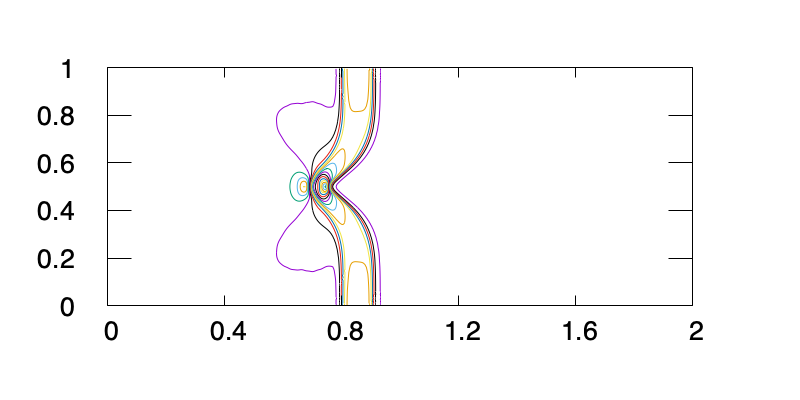}}
              }
              \mbox{ \subfigure[$t=0.36$]					
                   {\includegraphics[width=7cm]{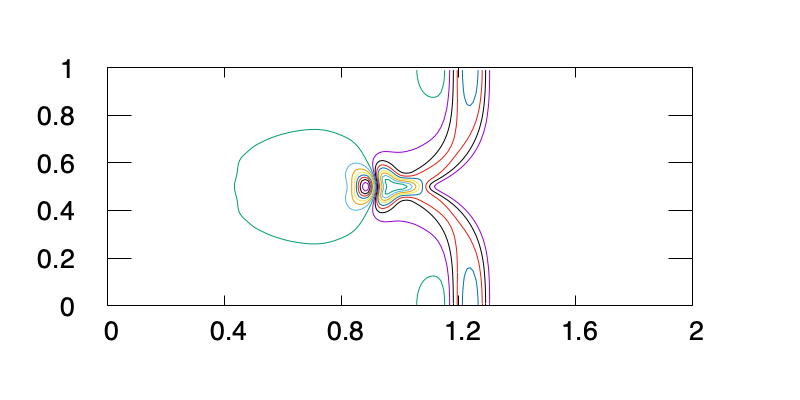}}\quad		
                     \subfigure[$t=0.36$]		
             {\includegraphics[width=7cm]{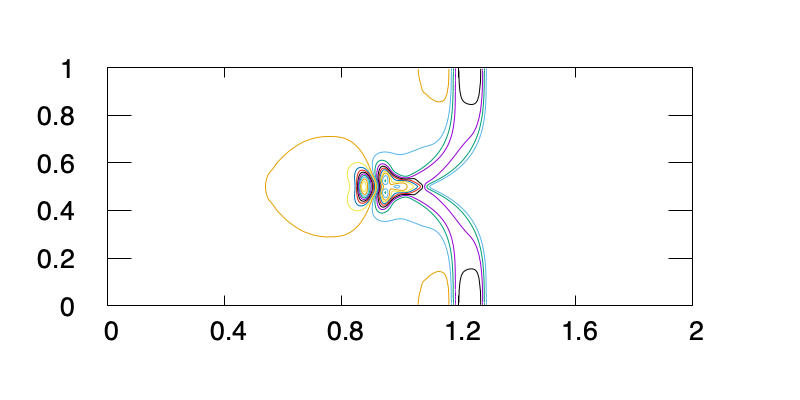}}
              }
              \mbox{\subfigure[$t=0.48$]					
                   {\includegraphics[width=7cm]{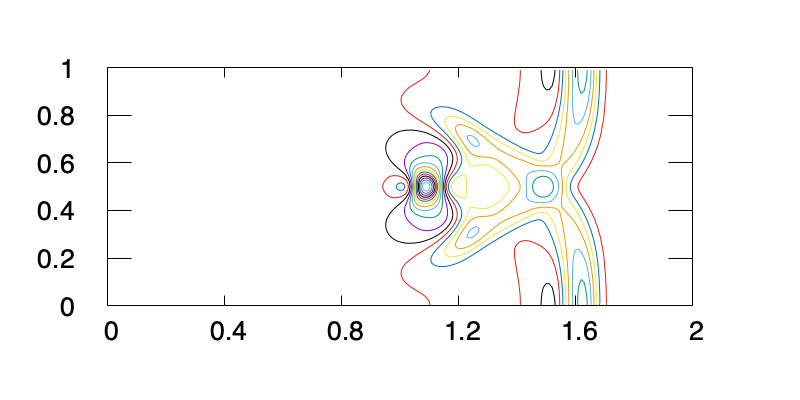}}\quad		
                    \subfigure[$t=0.48$]		
             {\includegraphics[width=7cm]{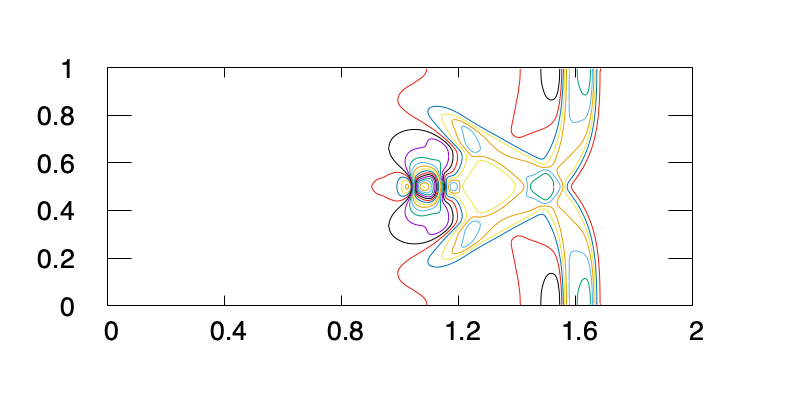}}
              }
              \mbox{\subfigure[$t=0.6$]					
                   {\includegraphics[width=7cm]{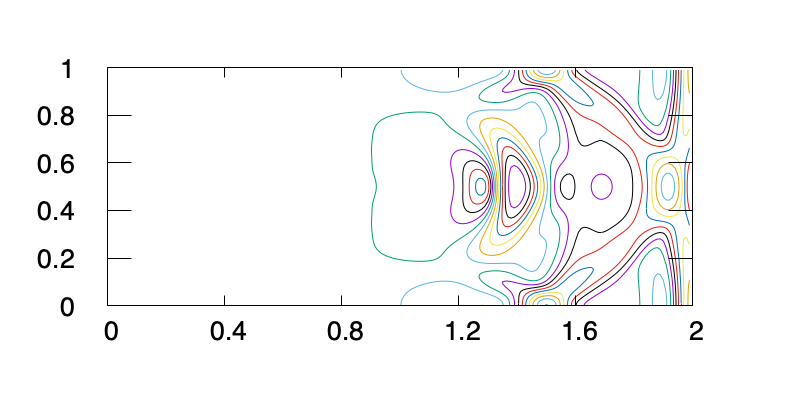}}\quad		
                    \subfigure[$t=0.6$]		
             {\includegraphics[width=7cm]{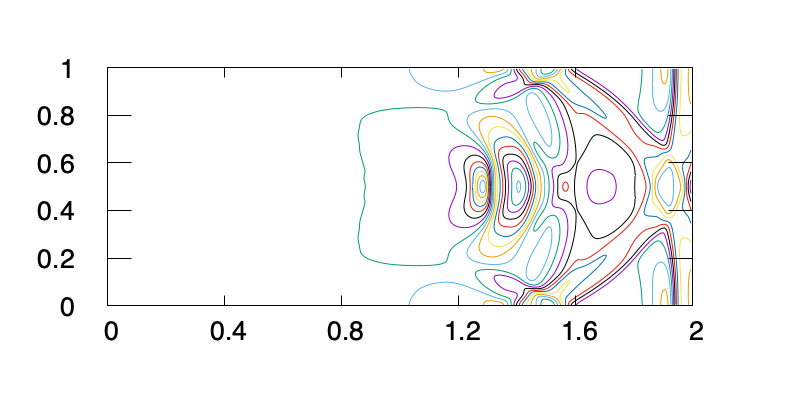}}
              }
              \caption{ Example~\ref{exam2D_2}. Numerical solutions of the surface level $H$.
      From top to bottom: at $t=0.12$ from 0.9998 to 1.0060; at  $t=0.24$ from 0.9967 to 1.0130;
      at  $t=0.36$ from 0.9901 to 1.0097; $t=0.48$ from 0.9906 to 1.0043; $t=0.6$ from 0.9955 to 1.0045, 30 contour lines are used. Mesh: $200\times100$ (left);  $400\times200$ (right). }
              \label{Fig_ex2D_2_1}
            \end{center}
          \end{figure}

      }
      \end{exa}

\section{Conclusion}
\label{sec6}
\setcounter{equation}{0}
\setcounter{figure}{0}

In this paper, we have proposed a high order AP finite difference WENO scheme for the SWEs with Manning friction and a bottom topology. The scheme is proved to be AP, AA and well-balanced, which are also confirmed by one- and two-dimensional numerical tests. In this work, the diffusion limit are added as a penalization, and an additive IMEX RK method has been utilized. In the future, we may further consider a semi-implicit approach without adding the diffusion terms.

\section*{Acknowledgements}

Guanlan Huang and Tao Xiong are partially supported by National Key R\&D Program of China No. 2022YFA1004500, NSFC grant No. 11971025 and 92270112, NSF of Fujian Province No. 2023J02003, and the Strategic Priority Research Program of Chinese Academy of Sciences Grant No. XDA25010401.

Sebastiano Boscarino is partially supported by the Italian Ministry of Instruction, University and Research (MIUR), with funds coming from PRIN Project 2022 (2022KA3JBA, entitled “Advanced numerical methods for time dependent parametric partial differential equations and applications”) and from Italian Ministerial grant PRIN 2022 PNRR “FIN4GEO: Forward and Inverse Numerical Modeling of hydrothermal systems in volcanic regions with application to geothermal energy exploitation.”, No. P2022BNB97.  Furthermore, Sebastiano Boscarino is supported for this work by the Spoke 1 ”FutureHPC \& BigData” of the Italian Research Center on High-Performance Computing, Big Data and Quantum Computing (ICSC) funded by MUR "Missione 4 Componente 2 Investimento 1.4: Potenziamento strutture di ricerca e creazione di campioni nazionali di R\&S (M4C2-19 )”. S. Boscarino is a member of the INdAM Research group GNCS.

\bibliographystyle{abbrv}
\bibliography{refer}

\newcommand{\noop}[1]{}
\begin{thebibliography}{10}

\bibitem{andreu1999existence}
F.~Andreu, J.~Mazon, S.~de~Le{\'o}n, and J.~Toledo.
\newblock Existence and uniqueness for a degenerate parabolic equation with
  ${L}^1$-data.
\newblock {\em Transactions of the American Mathematical Society},
  351(1):285--306, 1999.

\bibitem{ascher1997implicit}
U.~M. Ascher, S.~J. Ruuth, and R.~J. Spiteri.
\newblock Implicit-explicit {R}unge-{K}utta methods for time-dependent partial
  differential equations.
\newblock {\em Applied Numerical Mathematics}, 25(2):151--167, 1997.

\bibitem{barrett1994finite}
J.~W. Barrett and W.~Liu.
\newblock Finite element approximation of the parabolic p-{L}aplacian.
\newblock {\em SIAM Journal on Numerical Analysis}, 31(2):413--428, 1994.

\bibitem{barrett1993finite}
J.~W. Barrett and W.~B. Liu.
\newblock Finite element approximation of the p-{L}aplacian.
\newblock {\em Mathematics of Computation}, 61(204):523--537, 1993.

\bibitem{bermudez1994upwind}
A.~Bermudez and M.~E. Vazquez.
\newblock Upwind methods for hyperbolic conservation laws with source terms.
\newblock {\em Computers \& Fluids}, 23(8):1049--1071, 1994.

\bibitem{BLeFT}
C.~Berthon, P.~LeFloch, and R.~Turpault.
\newblock Late-time/stiff-relaxation asymptotic-preserving approximations of
  hyperbolic equations.
\newblock {\em Mathematics of Computation}, 82(282):831--860, 2013.

\bibitem{boscarino2016high}
S.~Boscarino, F.~Filbet, and G.~Russo.
\newblock High order semi-implicit schemes for time dependent partial
  differential equations.
\newblock {\em Journal of Scientific Computing}, 68(3):975--1001, 2016.

\bibitem{boscarino2014high}
S.~Boscarino, P.~G. LeFloch, and G.~Russo.
\newblock {High-order asymptotic-preserving methods for fully nonlinear
  relaxation problems}.
\newblock {\em SIAM Journal on Scientific Computing}, 36(2):A377--A395, 2014.

\bibitem{boscarino2013implicit}
S.~Boscarino, L.~Pareschi, and G.~Russo.
\newblock {Implicit-explicit Runge--Kutta schemes for hyperbolic systems and
  kinetic equations in the diffusion limit}.
\newblock {\em SIAM Journal on Scientific Computing}, 35(1):A22--A51, 2013.

\bibitem{boscarino2019high}
S.~Boscarino, J.-M. Qiu, G.~Russo, and T.~Xiong.
\newblock {A high order semi-implicit IMEX WENO scheme for the all-Mach
  isentropic Euler system}.
\newblock {\em Journal of Computational Physics}, 392:594--618, 2019.

\bibitem{boscarino2022high}
S.~Boscarino, J.-M. Qiu, G.~Russo, and T.~Xiong.
\newblock High order semi-implicit {WENO} schemes for all {M}ach full {E}uler
  system of gas dynamics.
\newblock {\em SIAM Journal on Scientific Computing}, 4:B368--B394, 2022.

\bibitem{boscarino2013flux}
S.~Boscarino and G.~Russo.
\newblock Flux-explicit {IMEX} {Runge-Kutta} schemes for hyperbolic to
  parabolic relaxation problems.
\newblock {\em SIAM Journal on Numerical Analysis}, 51(1):163--190, 2013.

\bibitem{bulteau2020fully}
S.~Bulteau, M.~Badsi, C.~Berthon, and M.~Bessemoulin-Chatard.
\newblock A fully well-balanced and asymptotic preserving scheme for the
  shallow-water equations with a generalized {M}anning friction source term.
\newblock {\em Calcolo}, 58(4):41, 2021.

\bibitem{butcher2016}
J.~C. Butcher.
\newblock {\em Numerical Methods for Ordinary Differential Equations}.
\newblock Third Edition, John Wiley \& Sons Ltd, 2016.

\bibitem{CGP2006}
M.~Castro, J.~Gallardo, and C.~Par{\'e}s.
\newblock {High order finite volume schemes based on reconstruction of states
  for solving hyperbolic systems with nonconservative products. Applications to
  shallow-water systems}.
\newblock {\em Mathematics of Computation}, 75(255):1103--1134, 2006.

\bibitem{CLL}
G.-Q. Chen, C.~D. Levermore, and T.-P. Liu.
\newblock Hyperbolic conservation laws with stiff relaxation terms and entropy.
\newblock {\em Communications on Pure and Applied Mathematics}, 47(6):787--830,
  1994.

\bibitem{chertock2015well}
A.~Chertock, S.~Cui, A.~Kurganov, and T.~Wu.
\newblock Well-balanced positivity preserving central-upwind scheme for the
  shallow water system with friction terms.
\newblock {\em International Journal for Numerical Methods in Fluids},
  78(6):355--383, 2015.

\bibitem{Crispo2019global}
F.~Crispo, P.~Maremonti, and M.~Růžička.
\newblock {Global $L^r$-estimates and regularizing effect for solutions to the
  $p(t,x)$-Laplacian systems}.
\newblock {\em Advances in Differential Equations}, 24(7/8):407 -- 434, 2019.

\bibitem{de1999regularity}
S.~S. de~Le{\'o}n and J.~Toledo.
\newblock Regularity for entropy solutions of parabolic p-{L}aplacian type
  equations.
\newblock {\em Publicacions Matematiques}, pages 665--683, 1999.

\bibitem{diaz1981estimates}
J.~I. D{\'\i}az and M.~Herrero.
\newblock Estimates on the support of the solutions of some nonlinear elliptic
  and parabolic problems.
\newblock {\em Proceedings of the Royal Society of Edinburgh Section A:
  Mathematics}, 89(3-4):249--258, 1981.

\bibitem{eom2020large}
J.~Eom and R.~Sato.
\newblock Large time behavior of ode type solutions to parabolic p-laplacian
  type equations.
\newblock {\em Communications on Pure \& Applied Analysis}, 19(9), 2020.

\bibitem{ern2008well}
A.~Ern, S.~Piperno, and K.~Djadel.
\newblock A well-balanced {R}unge--{K}utta discontinuous {G}alerkin method for
  the shallow-water equations with flooding and drying.
\newblock {\em International Journal for Numerical Methods in Fluids},
  58(1):1--25, 2008.

\bibitem{garcia1987existence}
J.~Garc{\'\i}a~Azorero and I.~Peral~Alonso.
\newblock Existence and nonuniqueness for the p-{L}aplacian.
\newblock {\em Communications in Partial Differential Equations},
  12(12):126--202, 1987.

\bibitem{gomez2021collocation}
I.~G{\'o}mez-Bueno, M.~J.~C. D{\'\i}az, C.~Par{\'e}s, and G.~Russo.
\newblock Collocation methods for high-order well-balanced methods for systems
  of balance laws.
\newblock {\em Mathematics}, 9(15):1799, 2021.

\bibitem{hairer1993solving2}
E.~Hairer and G.~Wanner.
\newblock {\em Solving ordinary differential equations II: stiff and
  differential algebraic problems}, volume~2.
\newblock Springer Verlag, 1993.

\bibitem{herrero1981asymptotic}
M.~A. Herrero and J.~L. V{\'a}zquez.
\newblock Asymptotic behaviour of the solutions of a strongly nonlinear
  parabolic problem.
\newblock In {\em Annales de la Facult{\'e} des sciences de Toulouse:
  Math{\'e}matiques}, volume~3, pages 113--127, 1981.

\bibitem{hu2017ap}
J.~Hu, S.~Jin, and Q.~Li.
\newblock Asymptotic-preserving schemes for multiscale hyperbolic and kinetic
  equations.
\newblock {\em Handbook of Numerical Analysis}, 18:103--129, 2017.

\bibitem{huang2022high}
G.~Huang, Y.~Xing, and T.~Xiong.
\newblock High order well-balanced asymptotic preserving finite difference
  {WENO} schemes for the shallow water equations in all {F}roude numbers.
\newblock {\em Journal of Computational Physics}, 463:111255, 2022.

\bibitem{jiang1996efficient}
G.-S. Jiang and C.-W. Shu.
\newblock Efficient implementation of weighted {ENO} schemes.
\newblock {\em Journal of Computational Physics}, 126(1):202--228, 1996.

\bibitem{jin1999efficient}
S.~Jin.
\newblock Efficient asymptotic-preserving ({AP}) schemes for some multiscale
  kinetic equations.
\newblock {\em SIAM Journal on Scientific Computing}, 21(2):441--454, 1999.

\bibitem{jin2022asymptotic}
S.~Jin.
\newblock Asymptotic-preserving schemes for multiscale physical problems.
\newblock {\em Acta Numerica}, 31:415--489, 2022.

\bibitem{kamin1988fundamental}
S.~Kamin and J.~L. V{\'a}zquez.
\newblock Fundamental solutions and asymptotic behaviour for the $ p
  $-{L}aplacian equation.
\newblock {\em Revista Matem{\'a}tica Iberoamericana}, 4(2):339--354, 1988.

\bibitem{Kurganov18}
A.~Kurganov.
\newblock Finite-volume schemes for shallow-water equations.
\newblock {\em Acta Numerica}, 27:289--351, 2018.

\bibitem{kurganov2002central}
A.~Kurganov and D.~Levy.
\newblock Central-upwind schemes for the {S}aint-{V}enant system.
\newblock {\em ESAIM: Mathematical Modelling and Numerical Analysis},
  36(3):397--425, 2002.

\bibitem{leveque1998balancing}
R.~J. LeVeque.
\newblock Balancing source terms and flux gradients in high-resolution
  {G}odunov methods: the quasi-steady wave-propagation algorithm.
\newblock {\em Journal of Computational Physics}, 146(1):346--365, 1998.

\bibitem{michel2017well}
V.~Michel-Dansac, C.~Berthon, S.~Clain, and F.~Foucher.
\newblock A well-balanced scheme for the shallow-water equations with
  topography or {M}anning friction.
\newblock {\em Journal of Computational Physics}, 335:115--154, 2017.

\bibitem{noelle2006well}
S.~Noelle, N.~Pankratz, G.~Puppo, and J.~R. Natvig.
\newblock Well-balanced finite volume schemes of arbitrary order of accuracy
  for shallow water flows.
\newblock {\em Journal of Computational Physics}, 213(2):474--499, 2006.

\bibitem{peng2020stability}
Z.~Peng, Y.~Cheng, J.-M. Qiu, and F.~Li.
\newblock {Stability-enhanced AP IMEX-LDG schemes for linear kinetic transport
  equations under a diffusive scaling}.
\newblock {\em Journal of Computational Physics}, 415:109485, 2020.

\bibitem{rhebergen2008discontinuous}
S.~Rhebergen, O.~Bokhove, and J.~J. van~der Vegt.
\newblock Discontinuous galerkin finite element methods for hyperbolic
  nonconservative partial differential equations.
\newblock {\em Journal of Computational Physics}, 227(3):1887--1922, 2008.

\bibitem{shu1998essentially}
C.-W. Shu.
\newblock Essentially non-oscillatory and weighted essentially non-oscillatory
  schemes for hyperbolic conservation laws.
\newblock {\em Advanced Numerical Approximation of Nonlinear Hyperbolic
  Equations}, pages 325--432, 1998.

\bibitem{song2011unstructured}
L.~Song, J.~Zhou, Q.~Li, X.~Yang, and Y.~Zhang.
\newblock An unstructured finite volume model for dam-break floods with wet/dry
  fronts over complex topography.
\newblock {\em International Journal for Numerical Methods in Fluids},
  67(8):960--980, 2011.

\bibitem{wu2021uniformly}
K.~Wu and Y.~Xing.
\newblock {Uniformly high-order structure-preserving discontinuous Galerkin
  methods for Euler equations with gravitation: Positivity and
  well-balancedness}.
\newblock {\em SIAM Journal on Scientific Computing}, 43(1):A472--A510, 2021.

\bibitem{xia2018new}
X.~Xia and Q.~Liang.
\newblock A new efficient implicit scheme for discretising the stiff friction
  terms in the shallow water equations.
\newblock {\em Advances in Water Resources}, 117:87--97, 2018.

\bibitem{xia2017efficient}
X.~Xia, Q.~Liang, X.~Ming, and J.~Hou.
\newblock An efficient and stable hydrodynamic model with novel source term
  discretization schemes for overland flow and flood simulations.
\newblock {\em Water Resources Research}, 53(5):3730--3759, 2017.

\bibitem{X2017}
Y.~Xing.
\newblock Numerical methods for the nonlinear shallow water equations.
\newblock In R.~Abgrall and C.-W. Shu, editors, {\em Handbook of Numerical
  Analysis, Volume 18, Handbook of Numerical Methods for Hyperbolic Problems:
  Applied and Modern Issues}, volume~18, pages 361--384. North-Holland,
  Elsevier, Amsterdam, 2017.

\bibitem{xing2005high}
Y.~Xing and C.-W. Shu.
\newblock High order finite difference {WENO} schemes with the exact
  conservation property for the shallow water equations.
\newblock {\em Journal of Computational Physics}, 208(1):206--227, 2005.

\bibitem{xing2006new}
Y.~Xing and C.-W. Shu.
\newblock {A new approach of high order well-balanced finite volume WENO
  schemes and discontinuous Galerkin methods for a class of hyperbolic systems
  with source terms}.
\newblock {\em Communications in Computational Physics}, 1(1):100--134, 2006.

\bibitem{xing2006}
Y.~Xing and C.-W. Shu.
\newblock High-order well-balanced finite difference weno schemes for a class
  of hyperbolic systems with source terms.
\newblock {\em Journal of Scientific Computing}, 27:477--494, 2006.

\bibitem{xing2006high}
Y.~Xing and C.-W. Shu.
\newblock {High order well-balanced finite volume WENO schemes and
  discontinuous Galerkin methods for a class of hyperbolic systems with source
  terms}.
\newblock {\em Journal of Computational Physics}, 214(2):567--598, 2006.

\bibitem{Xing14}
Y.~Xing and C.-W. Shu.
\newblock A survey of high order schemes for the shallow water equations.
\newblock {\em Journal of Mathematical Study}, 47:221--249, 2014.

\bibitem{xiong2022high}
T.~Xiong, W.~Sun, Y.~Shi, and P.~Song.
\newblock {High order asymptotic preserving discontinuous Galerkin methods for
  gray radiative transfer equations}.
\newblock {\em Journal of Computational Physics}, 463:111308, 2022.

\bibitem{yang2021high}
R.~Yang, Y.~Yang, and Y.~Xing.
\newblock High order sign-preserving and well-balanced exponential
  {R}unge-{K}utta discontinuous {G}alerkin methods for the shallow water
  equations with friction.
\newblock {\em Journal of Computational Physics}, 444:110543, 2021.

\bibitem{zhou2001surface}
J.~G. Zhou, D.~M. Causon, C.~G. Mingham, and D.~M. Ingram.
\newblock The surface gradient method for the treatment of source terms in the
  shallow-water equations.
\newblock {\em Journal of Computational Physics}, 168(1):1--25, 2001.

\end{thebibliography}
\end{document}